\documentclass[oneside, 12pt]{amsart}

\usepackage{comment}

\usepackage{mathtools, nccmath}

\usepackage{tikz}

\usepackage{enumerate}
\usepackage{array}
\usepackage[utf8]{inputenc}
\usepackage[margin=1in]{geometry}
\usepackage{amsthm}
\usepackage{enumitem}  
\usepackage[OT2,T1]{fontenc}
\usepackage{amscd,amssymb, enumitem, amsmath,mathrsfs, amsfonts, amsaddr}
\usepackage{url}
\usepackage{tikz}
\usepackage{fullpage}
\usepackage{microtype}
\usetikzlibrary{decorations.pathreplacing}
\usepackage[backref=page]{hyperref}
\usepackage{hyperref}
\usepackage[utf8]{inputenc}
\usepackage{enumitem}
\usepackage[new]{old-arrows}
\usepackage{mathtools}

\usepackage{caption}

\usetikzlibrary{decorations.pathreplacing,calligraphy}

\DeclareSymbolFont{cyrletters}{OT2}{wncyr}{m}{n}
\DeclareMathSymbol{\Sha}{\mathalpha}{cyrletters}{"58}

\makeatletter
\newcommand{\pushright}[1]{\ifmeasuring@#1\else\omit\hfill$\displaystyle#1$\fi\ignorespaces}
\newcommand{\pushleft}[1]{\ifmeasuring@#1\else\omit$\displaystyle#1$\hfill\fi\ignorespaces}
\makeatother

\usepackage[utf8]{inputenc}

\newtheorem{theorem}{Theorem}[section]
\newtheorem{lemma}[theorem]{Lemma}
\newtheorem{proposition}[theorem]{Proposition}
\newtheorem{proposition*}{Proposition}
\newtheorem{corollary}[theorem]{Corollary}
\newtheorem*{theorem*}{Theorem}
\newtheorem*{corollary*}{Corollary}

\theoremstyle{definition}

\newtheorem{example}[theorem]{Example}

\newtheorem{claim}[theorem]{Claim}
\newtheorem{remark}[theorem]{Remark}
\newtheorem{emp}[theorem]{\bf }{\kern -4pt}
\newtheorem*{acknowledgement}{Acknowledgements}

\theoremstyle{theorem}
\newtheorem{conjecture}[theorem]{Conjecture}
\theoremstyle{plain}

\newtheoremstyle{TheoremNum}
{\topsep}{\topsep}              
{\itshape}                      
{}                              
{\bfseries}                     
{.}                             
{ }                             
{\thmname{#1}\thmnote{ \bfseries #3}}
\theoremstyle{TheoremNum}

\newtheorem{prorep}{Proposition}

\theoremstyle{remark}

\DeclareRobustCommand\longtwoheadrightarrow
{\relbar\joinrel\twoheadrightarrow}

\newcommand{\bd}{400}

\newcommand{\reffourmidm}{9.13}
\newcommand{\refverification}{2.18}

\title{On the Kodaira types of Elliptic curves with potentially good supersingular reduction}
\author{Haiyang Wang}
\address{}
\date{\today}

\usepackage{draftwatermark}
\SetWatermarkText{}
\SetWatermarkScale{3}

\begin{document}
	
	\maketitle
	

	\begin{abstract}
		Let $\mathcal{O}_K$ be a Henselian discrete valuation domain with field of
		fractions $K$. Assume that $\mathcal{O}_K$ has algebraically closed residue field $k$. Let $E/K$ be an elliptic curve with additive reduction. The semi-stable
		reduction theorem asserts that there exists a minimal extension $L/K$ such that the base change $E_L/L$ has semi-stable reduction.
		
		It is natural to wonder whether specific properties of the semi-stable reduction and of the extension $L/K$ impose restrictions on what types of Kodaira type the special fiber of $E/K$ may have. In this paper we study the restrictions imposed on the reduction type
		when the extension $L/K$ is wildly ramified of degree $2$, and the curve $E/K$ has potentially good supersingular reduction. We also analyze the possible reduction types of two isogenous elliptic curves with these properties.
	\end{abstract}


\section{Introduction}

\begin{emp}\label{em.intro}
	Let $\mathcal{O}_K$ be a Henselian discrete valuation domain with field of fractions $K$. Assume that $\mathcal{O}_K$ has algebraically closed residue field $k$ of characteristic $2$. Let $\pi$ be a uniformizer and $v$ the normalized valuation of $K$ so that $v(\pi)=1$. We will also use $\pi_K$ and $v_K$ to denote a uniformizer and the normalized valuation of $K$ respectively when there might be confusion. We use the convention that $v(0)=\infty$. 
	
	Let $E/K$ be an elliptic curve. Assume that $E/K$ has bad reduction modulo $(\pi)$. By the semistable reduction theorem (see e.g. \cite{Liu} Theorem 10.4.44), there exists a unique minimal extension $L/K$ such that the base change $E_L/L$ has semistable reduction at the unique prime ideal above $(\pi)$. Here minimal means that every extension that acquires semistable reduction contains $L$, and the uniqueness of the prime ideal over $(\pi)$ comes from the assumption that $K$ is Henselian. 
	
	In \cite{Lorenzini10}, Theorem 2.8, Lorenzini investigated the reduction type of $E/K$ under the assumption that $E_L/L$ has multiplicative reduction. In \cite{Lorenzini13}, Theorem 4.2, Lorenzini determined the reduction type of $E/K$ assuming that $E_L/L$ has good ordinary reduction. In both cases, $L/K$ is quadratic and $E/K$ has the reduction type $\mathrm{I}^{\ast}_{n}$ for some $n\ge 1$.  
	
	In this article, we study the reduction type of $E/K$ when $E_L/L$ has good supersingular reduction. It is not true in general that $L/K$ is quadratic in this setting. See \cite{Kraus}, Théorème 2 and Théorème 3. We will restrict our attention to the case where $[L:K]=2$. Even under this assumption, it is not true that the reduction type can only be of the kind $\mathrm{I}^{\ast}_{n}$.
	
	Let $H:=\operatorname{Gal}(L/K)$ and $H= H_0\trianglerighteq H_1\trianglerighteq,\dots$ be the higher ramification groups of $L/K$. Here the equality $H= H_0$ follows from the assumption that the residue field $k$ is algebraically closed. Define $s_{L/K}$ by the relation $s_{L/K}+1=\sum\limits_{i=0}^{\infty}(|H_i|-1)$. Note that $\sum\limits_{i=0}^{\infty}(|H_i|-1)$ is the valuation of the different of the extension $\mathcal{O}_L/\mathcal{O}_K$. See \cite{Serre} IV.2, Proposition 4. Our first main theorem is the following.
	
\end{emp}

\medskip

\begin{theorem}\label{thm.supred}
	Let $\mathcal{O}_K$ be as in \ref{em.intro}. Let $E/K$ be an elliptic curve with additive reduction. Assume that there exists a quadratic extension $L/K$ such that $E_L/L$ has good supersingular reduction. Let $j:=j(E)$. Then the following is true.
	
	\begin{enumerate}[label=(\alph*)]

		\item Assume that $v(j)\le 4s_{L/K}-4$. Then $E/K$ has reduction type $\mathrm{I}^{\ast}_{4s_{L/K}-v(j)}$ with $v(j)$ divisible by $12$.
		
		\item Assume that $v(j)> 4s_{L/K}-4$. Let $f\in\{1,3,5\}$ satisfy $2s_{L/K}+3\equiv f \operatorname{mod} 6$. Then $E/K$ has reduction type $\mathrm{II} \,\,(\text{resp. } \mathrm{I}^{\ast}_0, \mathrm{II}^{\ast})$ if and only if $f=1\,\,(\text{resp. 3,5})$.
		
	\end{enumerate}
\end{theorem}
\medskip
The proof of Theorem \ref{thm.supred} is found in \ref{em.proof.red}.
\begin{remark}
	Let $\mathcal{O}_K$ be as in \ref{em.intro}. Let $E/K$ be an elliptic curve with additive reduction. Assume that there is a quadratic extension $L/K$ such that $E_L/L$ has good supersingular reduction. Then $E_L/L$ has a smooth model $\mathcal{Y}/\mathcal{O}_L$. The Galois group $H:=\operatorname{Gal}(L/K)$ acts on $\mathcal{Y}$ and $\mathcal{Y}_k$. Let $\mathcal{Z}/\mathcal{O}_K$ be the quotient $\mathcal{Y}/G$. Let $\sigma$ denote the automorphism of $\mathcal{Y}_k$ induced by the action of the generator of $G$. Then $\sigma$ is the involution of $\mathcal{Y}_k$. Since $\mathcal{Y}_k$ is assumed to be supersingular, $\mathcal{Y}_k$ has no non-trivial $2$-torsion. Hence $\sigma$ has a unique fixed point $P$. The scheme $\mathcal{Z}$ is singular exactly at the image of $P$, denote it by $Q$, under the quotient map $\mathcal{Y}\rightarrow\mathcal{Z}$. See \cite{Lorenzini14} p. 331-332 for more details about this construction. Assume that $E/K$ has Kodaira type $\mathrm{I}^{\ast}_{m}$ with $m\ge 1$. The minimal resolution of $Q$ has dual graph of the following shape:
	
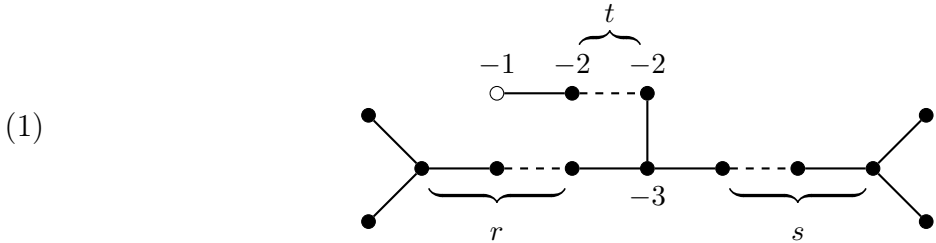
\begin{equation}
		\label{fig:dualgraph}
		\begin{gathered}
			\begin{tikzpicture}
				[node distance=1cm, font=\small]
				\tikzstyle{vertex}=[circle, draw, fill, inner sep=0mm, minimum size=1.0ex]
				\node[vertex]	(H1)  	at (0,0) 	[label=below:{}]               {};
				\node[vertex]	(H2)  	         	[right of=H1,label=below:{}]               {};
				\node[vertex]	(H3)  	         	[right of=H2,label=below:{}]               {};
				\node[vertex]	(HC)			[right of=H3, label= below:{$-3$ }]    {};
				\node[vertex]	(H4)			[right of=HC, label=above:{}]    {};
				\node[vertex]	(H5)			[right of=H4, label=below:{}]    {};
				\node[vertex]	(H6)			[right of=H5, label=above:{}]    {};
				\node[vertex]	(W1)			[above left of=H1, label=above:{}]    {};
				\node[vertex]	(W2)			[below left of=H1, label=above:{}]    {};
				\node[vertex]	(E1)			[above right of=H6, label=below:{}]    {};
				\node[vertex]	(E2)			[below right of=H6, label=below:{}]    {};
				\node[vertex]	(V1)			[above of=HC, label=above:{$-2$}]    {};
				\node[vertex]	(V2)			[left of=V1, label=above:{$-2$}]    {};
				\node[vertex,fill=none]	(V3)			[left of=V2, label=above:{$-1$}]    {};
				\draw [thick] (H1)--(H2);
				\draw [thick,dashed]  (H2)--(H3);
				\draw [thick] (H3)--(HC)--(H4);
				\draw [thick,dashed]  (H4)--(H5);
				\draw [thick] (H5)--(H6);
				\draw [thick] (W1)--(H1)--(W2);
				\draw [thick] (E1)--(H6)--(E2);
				\draw [thick] (V1)--(HC);
				\draw [thick,dashed] (V2)--(V1);
				\draw [thick] (V2)--(V3);
				\draw [decorate,line width=1.25pt, decoration={calligraphic brace,mirror,amplitude=5pt,raise=3.6ex,}]
				(V1) -- (V2) node[midway,xshift=0em, yshift=2.6em]{$t$};
				\draw [decorate,line width=1.25pt, decoration={calligraphic brace,amplitude=5pt,mirror,raise=1.5ex,}]
				(H1) -- (H3) node[midway,xshift=0em, yshift=-2.1em]{$r$};
				\draw [decorate,line width=1.25pt, decoration={calligraphic brace,amplitude=5pt,mirror,raise=1.5ex,}]
				(H4) -- (H6) node[midway,xshift=0em, yshift=-2.1em]{$s$};
			\end{tikzpicture}
		\end{gathered}
	\end{equation}
	for some integers $r,t\ge 0$ and $s\ge 1$ with $r+s=m$. Here the white vertex corresponds to the component $\mathcal{Z}_k$.

	Conversely,	let $r=2r'$, $s=2s'$, and $t=2t'+1$, where $r', t'\ge 0$, and $s'\ge 1$ are integers. Let $\Gamma$ be the graph with black vertices in \eqref{fig:dualgraph} associated to $r,s,t$ (remove the white vertex and the edge connected to it). Section $\reffourmidm$ in \cite{Lorenzini24} suggests that
	$\Gamma$ could be the dual graph of the intersection matrix associated with the resolution of a wild $\mathbb{Z}/2\mathbb{Z}$ quotient surface singularity. It is interesting to note that Theorem \ref{thm.supred} shows that such resolution of singularity occurs in the elliptic curve setting described above only when $r'+s'$ is even, since $r+s=4s_{L/K}-v(j)$ has to be divisible by $4$ in this case.
\end{remark}

Let $\mathcal{O}_K$ be as in \ref{em.intro}. Let $E_1/K$ and $E_2/K$ be elliptic curves with additive reduction. Assume that there is an isogeny of degree $d$ from $E_1/K$ to $E_2/K$. If $E_1/K$ and $E_2/K$ have potentially good ordinary reduction, Corollary 4.6 in \cite{Lorenzini13} shows that $E_1/K$ and $E_2/K$ have the same reduction type. In Corollary \ref{cor.sameIn}, we show that if $E_1/K$ and $E_2/K$ have potentially good supersingular reduction with $[L:K]=2$ and $v(2)=1$, then $E_1/K$ and $E_2/K$ have the same reduction type. However, if $v(2)>1$, $E_1/K$ and $E_2/K$ do not necessarily have the same reduction type (see Remark \ref{rk.counteregisog}). 

When $K$ is a finite extension of $\mathbb{Q}_2$ and $d$ is odd, T. Dokchitser and V. Dokchitser (\cite{Dokchitser-Dokchitser}, Theorem 5.4 (1)) showed that $E_1/K$ and $E_2/K$ have the same reduction type. When $d=2$, we prove the following.

\medskip

\begin{prorep}[\ref{prop.sametype2}]
	Let $\mathcal{O}_K$ be as in \ref{em.intro} with mixed characteristic $2$. Let $E_1/K$ and $E_2/K$ be elliptic curves and assume that there exists a $2$-isogeny from $E_1/K$ to $E_2/K$. Assume that $E_1/K$ and $E_2/K$ have additive reduction and good supersingular reduction after a quadratic extension $L/K$. Then $E_1/K$ and $E_2/K$ have the same Kodaira type if and only if either $\operatorname{min}(v(j(E_1)),v(j(E_2)))>4s_{L/K}-4$ or $v(j(E_1))=v(j(E_2))=6v(2)$.
	
\end{prorep}

\medskip

The reduction type under isogeny is further discussed in Section \ref{sec.isog}.

\begin{acknowledgement}
	This article forms part of the author’s doctoral thesis at the University of Georgia. The author would like to thank his advisor, Dino Lorenzini, 
	for many helpful suggestions and encouragement.
\end{acknowledgement}

\section{The reduction type}

In this section, our main goal is to establish Theorem \ref{thm.supred}. We begin by proving several preliminary results. Let $K$ be as in \ref{em.intro}. We recall the following facts about $s_{L/K}$ when $L/K$ is quadratic. Let $\sigma$ denote the generator of $\operatorname{Gal}(L/K)$.

\begin{lemma}\label{le.sl_k}
	Let $\mathcal{O}_K$ be as in \ref{em.intro}. Let $L/K$ be a quadratic extension.

	\begin{enumerate}[label=(\alph*)]
		\item Assume that $\operatorname{char}(K)=2$. Then $L/K$ is the splitting field of an Artin-Schreier polynomial $f(z)=z^2+z+D$ with $v_K(D)<0$ and odd. Moreover, $s_{L/K}=-v_K(D)$.
		\item Assume that $\operatorname{char}(K)=0$. Then $L=K(\sqrt{D})$ for $D\in K$ with $v_K(D)=0$ or $v_K(D)=1$. If $v_K(D)=1$, then $s_{L/K}=2v_K(2)$. If $v_K(D)=0$, then $L$ is the splitting field of an Eisenstein polynomial $g(z)=z^2+az+b\in K[z]$ where $1\le v_K(a)$, $v_K(b)=1$ and $a^2-4b=Dc^2$ for some $c\in \mathcal{O}_K^*$ with $v_K(c)=v_K(a)\le v_K(2)$. Moreover, $s_{L/K}=2v_K(a)-1$.
	\end{enumerate}
	
\end{lemma}

\begin{proof}
	(a) The quadratic extension $L/K$ can be given by an Artin-Schreier polynomial $f(z)=z^2+z+D$ since $\operatorname{char}(K)=2$. We know that $r:=-v_K(D)>0$ since the residue field is algebraically closed and hence $f(z)$ is reducible in $\mathcal{O}_K[z]$ when $r\le 0$ by Hensel's Lemma. Now we show that $r$ can be assumed to be odd. Suppose that $D=\frac{a}{\pi^{2s}}$ with $v_K(a)=0$. Since the residue field is assumed to be algebraic closed, we can choose $b\in \mathcal{O}_K$ such that $v_K(b^2+a)>0$. Make the change of variable $z=Z+\frac{b}{\pi^s}$ to $f(z)$, and we get that $Z^2+Z+\frac{b\pi^s+b^2+a}{\pi^{2s}}$ with $v_K(\frac{b\pi^s+b^2+a}{\pi^{2s}})>-2s$. If $v_K(\frac{b\pi^s+b^2+a}{\pi^{2s}})$ is done, we are done. Otherwise, repeat the process, and we will eventually get a polynomial such that the constant term has negative odd valuation.
	
	Let $v_L$ be the normalized valuation in $L$. In particular, we have that $v_L(c)=2v_K(c)$ for each $c\in K$. Let $\alpha$ be a root of $f(z)$. Then $v_L(\alpha)=-r$. Because $r$ is odd, there exist $m,n\in\mathbb{N}_{\ge 1}$ such that $2m-rn=1$. Notice that $n$ is odd. Take $\pi_L:=\pi_K^m\alpha^n$. Then $v_L(\pi_L)=1$ and $\pi_L$ is a uniformizer of $\mathcal{O}_L$. It follows that $\sigma(\pi_L)-\pi_L=\pi_K^m(\alpha+1)^n-\pi_K^m\alpha^n=\pi_K^m({n\choose 1}\alpha^{n-1}+{n\choose 2}\alpha^{n-2}+\cdots +1)$. The term ${n\choose i}$ either is $0$ in $K$ or has valuation $0$ for $1\le i\le n$. Because $n$ is odd, the term ${n\choose 1}=n$ has valuation $0$ and so the sum has valuation $(n-1)v_L(\alpha)$. Hence, $v_L(\sigma(\pi_L)-\pi_L)=2m-(n-1)r=r+1$. Therefore $s_{L/K}=-1+\sum\limits_{i=0}^{\infty}(|H_i|-1)=r$.
	
	(b) When $\operatorname{char}(K)=0$, we know that $L=K(\sqrt{D})$ for some $D\in K$. Moreover, we can assume that $v(D)=0$ or $1$. Indeed, if $v(D)$ is even, then $L=K(\sqrt{D\pi^{-v(D)}})$. If $v(D)$ is odd, then $L=K(\sqrt{D\pi^{-v(D)+1}})$. 
	
	If $v_K(D)=1$, then $s_{L/K}=2v_K(2)$. Indeed, $\sqrt{D}$ is a uniformizer in $\mathcal{O}_L$. Let $v_L$ be the normalized valuation of $L$. Then $v_L(\sigma(\sqrt{D})-\sqrt{D})=v_L(-2\sqrt{D})=2v_K(2)+1$. So $s_{L/K}=-1+\sum\limits_{i=0}^{\infty}(|H_i|-1)=2v_K(2)$.

	If $v_K(D)=0$, the extension $L/K$ can be given by an Eisenstein polynomial $g(z)=z^2+az+b$ with $v_K(a)\ge 1$ and $v_K(b)=1$, because $L/K$ is totally ramified. It follows that $a^2-4b=Dc^2$ for some $c\in K$ with $v_K(c)\ge 0$. Then $\operatorname{min}(2v(a),v(b)+2v(2))$ has to be $2v(a)$, so $v_K(a)=v_K(c)\le v_K(2)$.
	
	Let $\pi_L$ be a root of $g(z)$. Then $\pi_L$ is a uniformizer for $\mathcal{O}_L$. It follows that $v_L(\sigma(\pi_L)-\pi_L)=v_L(-2\pi_L-a)=v_L(a)=2v_K(a)$. So $s_{L/K}=-1+\sum\limits_{i=0}^{\infty}(|H_i|-1)=2v_K(a)-1$.	
\end{proof}

\begin{lemma}\label{le.slkbound}
	Let $\mathcal{O}_K$ be as in \ref{em.intro} with mixed characteristic $2$. Let $L/K$ be a quadratic extension. If $s_{L/K}$ is even, then $s_{L/K}=2v(2)$. If $s_{L/K}$ is odd, then $1\le s_{L/K}\le 2v(2)-1$. Conversely, let $1\le s\le 2v(2)-1$ with $s$ odd or $s=2v(2)$. Then there exists a quadratic extension $L/K$ such that $s_{L/K}=s$.
\end{lemma}

\begin{proof}
	Lemma \ref{le.sl_k} (b) shows that if $s_{L/K}$ is even, then $s_{L/K}=2v(2)$ and that if $s_{L/K}$ is odd, then $s_{L/K}=2v(a)-1\le 2v(2)-1$, where the element $a$ is as defined in the lemma. If $s=2v(2)$, let $L_s$ be the splitting field of $z^2-\pi_K$. If $1\le s\le 2v(2)-1$ and $s$ is odd, let $L_s$ be the splitting field of $z^2+\pi_K^{(s+1)/2}z+\pi_K$. In both cases, Lemma \ref{le.sl_k} (b) shows that we have that $s_{L_s/K}=s$. 
\end{proof}

\begin{lemma}\label{le.supersing-va1-v2}
	Let $\mathcal{O}_K$ be as in \ref{em.intro}. Let $E/K$ be an elliptic curve with a minimal Weierstrass equation $y^2+a_1 xy+a_3 y=x^3+a_2x^2+a_4x+a_6$ where $a_i\in\mathcal{O}_K$. Let $j:=j(E)$. Then the following is true.

	\begin{enumerate}[label=(\alph*)]
		\item The elliptic curve $E/K$ has good supersingular reduction if and only if $a_3\in \mathcal{O}_K^*$ and $a_1\in \pi\mathcal{O}_K$. 
		\item Assume that $E/K$ has good reduction and $v(j)< 12v(2)$. Then $v(a_1)< v(2)$ and $v(j)=12v(a_1)$.
	\end{enumerate}

\end{lemma}

\begin{proof}
	(a) Let $b_i,c_i,\Delta$ and $j$ be the elements in $K$ associated to the given Weierstrass equation as discussed in \cite{SilvermanArith} III.1, page 42. Assume that $E/K$ has good supersingular reduction. Since $\operatorname{char}(k)=2$, we can deduce that $v(j)>0$. Because $E/K$ has good reduction, $v(\Delta)=0$. Hence $v(c_4)=v(j\Delta)/3>0$. It follows that $v(b_2)=v(c_4+24b_4)/2>0$ and therefore $v(a_1)=v(b_2-4a_2)/2>0$. Because $v(\Delta)=0$, it follows that $v(b_6)=v(-b_2^2 b_8-8b_4^3+9b_2 b_4 b_6-\Delta)/2=0$. Hence $v(a_3)=v(b_6-4a_6)/2=0$.
	
	Conversely, assume that $a_3\in \mathcal{O}_K^*$ and $a_1\in \pi\mathcal{O}_K$. 
	It follows that $v(b_2)=v(a_1^2+4a_2)>0$ and $v(b_6)=v(a_3^2+4a_6)=0$. Hence $v(c_4)=v(b_2^2-24b_4)>0$ and $v(\Delta)=v(-b_2^2 b_8-8b_4^3-27b_6^2+9b_2 b_4 b_6)=0$. Thus $v(j)=3v(c_4)-v(\Delta)>0$. Therefore $E/K$ has good supersingular reduction.
	
	(b) Assume that $v(a_1)\ge v(2)$. Then $v(j)=3v(c_4)=3v(a_1^4+8a_1^2a_2+16a_2^2-48a_4-24a_1a_3)\ge 12v(2)$, a contradiction. So $v(a_1)< v(2)$. It follows that $v(j)=3v(c_4)=12v(a_1)$.
\end{proof}

\begin{remark}\label{rk.gt12}
	Let $\mathcal{O}_K$ be a discrete valuation ring with quotient field of characteristic zero and algebraically closed residue field of characteristic $\ell>0$. Let $v$ be the normalized valuation on $K$. Let $E/K$ and $E'/K$ be elliptic curves with good supersingular reduction and let $j:=j(E)$ and $j':=j(E')$. Corollary 2.4 in \cite{Gross-Zagier} implies that $v(j-j')\ge 12$ when $\ell=2$. Lemma \ref{le.supersing-va1-v2} shows that we actually have $v(j)\ge 12$ and $v(j')\ge 12$. Indeed, if $v(j)\ge 12v(2)$, then it is clear that $v(j)\ge 12$. Assume that $v(j)< 12v(2)$. Let $y^2+a_1 xy+a_3 y=x^3+a_2x^2+a_4x+a_6$ with $a_i\in\mathcal{O}_K$ be a minimal Weierstrass equation for $E/K$. Lemma \ref{le.supersing-va1-v2} (b) shows that $v(j)=12v(a_1)$ and Lemma \ref{le.supersing-va1-v2} (a) shows that $v(a_1)\ge 1$. Therefore $v(j)\ge 12$. Similarly, $v(j')\ge 12$.
\end{remark}

\begin{proposition}\label{prop.minext-twist}
	
	Let $\mathcal{O}_K$ be as in \ref{em.intro}. Let $E/K$ be an elliptic curve with additive reduction and good supersingular reduction after a quadratic extension $L/K$. Let $j:=j(E)$. Then the following is true.
	\begin{enumerate}[label=(\alph*)]
		\item Let $E'/K$ be the twist of $E/K$ by the quadratic extension $L/K$. Then $E'/K$ has good supersingular reduction.
		\item If $v(j)\le 4s_{L/K}-4$, then $v(j)<12v(2)$, $v(2)>1$ and $12\mid v(j)$.
	\end{enumerate}
\end{proposition}

\begin{proof}
	(a) We use the ideas in the proof of Lemma 2.6 in \cite{Melistas}. Consider the Weil restriction $A/K:=\operatorname{Res}_{L/K}(E_L/L)$. Then $A/K$ is an abelian variety of dimension $2$ and is isogenous over $K$ to $E/K\times E'/K$ where $E'/K$ is the quadratic twist of $E/K$ by the extension $L/K$ (see Theorem 4.5 in \cite{Mazur-Rubin-Siverberg}). Remark on p.179 of \cite{Milne} shows that the abelian rank of $A/K$ is $1$. Because abelian rank is an isogeny invariant (see \cite{Grothendieck}, IX, Cor. 2.2.7), the abelian rank of $E/K\times E'/K$ is $1$. Recall that taking the Néron model commutes with the product operation. So, the assumption that $E/K$ has additive reduction implies that $E'/K$ has good reduction. Since $E_L/L$ has good supersingular reduction, $E'/K$ has good supersingular reduction.

	(b) Let $E'/K$ be the twist of $E/K$ by the quadratic extension $L/K$. Then $E'/K$ has good supersingular reduction by part (a). Let $y^2+a_1 xy+a_3 y=x^3+a_2x^2+a_4x+a_6$ with $a_i\in\mathcal{O}_K$ be a minimal Weierstrass equation for $E'/K$. By Lemma \ref{le.slkbound}, the condition $v(j)\le 4s_{L/K}-4$ implies that $v(j)<12v(2)$. Lemma \ref{le.supersing-va1-v2} (a) and (b) show that $v(2)>v(a_1)\ge 1$ and $12\mid v(j)$. 
\end{proof}

\begin{emp}\label{em.proof.red}
	Now we are ready to prove Theorem \ref{thm.supred}.
\end{emp}		

\begin{remark}\label{rk.notation}
	In the proof of Theorem \ref{thm.supred}, we need several changes of variables to equations in the variables $x$ and $y$. By abuse of notation, we will continue to use $x$ and $y$ for the variables in the new equation.
\end{remark}

\begin{proof}[Proof of Theorem \ref{thm.supred}]

	Let $E'/K$ be the quadratic twist of $E/K$ by the extension $L/K$. By Proposition \ref{prop.minext-twist}, we know that $E'/K$ has good supersingular reduction. Assume that $E'/K$ has a minimal model
	\begin{equation}\label{eq.completeWE}
		y^2+a_1 xy+a_3 y=x^3+a_2x^2+a_4x+a_6
	\end{equation}
	with $a_i\in\mathcal{O}_K$. After a translation of $x$, we can assume that $a_2=0$ and that the new equation is still minimal (we use here that $\operatorname{char}(K)\ne 3$). Lemma \ref{le.supersing-va1-v2}(a) shows that $a_3\in \mathcal{O}_K^*$ and $a_1\in \pi\mathcal{O}_K$ since $E'/K$ has good supersingular reduction. In the following, we use Tate's algorithm (see \cite{Tate} p. 47-52, or \cite{SilvermanAd}, IV 9.4) to determine the reduction type of $E/K$.

	{\it We work with the equicharacteristic $2$ case first}. By Lemma \ref{le.sl_k} (a), the extension $L/K$ can be given by an Artin-Schreier polynomial $f(x)=x^2+x+D$ for some $D\in K$ such that $s_{L/K}=-v(D)>0$ and $s_{L/K}$ is odd. Then $E/K$ can be given by the following equation:
	\begin{equation}\label{eq.twequiv}
		y^2+a_1xy+a_3y=x^3+Da_1^2x^2+a_4x+a_6+Da_3^2
	\end{equation}
	(see \cite{Connell} 4.3, page 410). In cases (a) and (b) below, we will choose a certain $t$ and consider the following dilation of Equation \eqref{eq.twequiv}: 
	\begin{equation}\label{eq.minequi}
		y^2+a_1\pi^txy+a_3\pi^{3t}y=x^3+Da_1^2 \pi^{2t}x^2+a_4 \pi^{4t}x+(a_6+Da_3^2)\pi^{6t}.
	\end{equation}

	\begin{claim}\label{cl.equi}
		$v(j)\le 4s_{L/K}-4$ if and only if $s_{L/K}-3v(a_1)\ge 1$.
	\end{claim}
	\begin{proof}[Proof of Claim \ref{cl.equi}]
		Since $E'/K$ has good reduction and $\operatorname{char}(K)=2$, we get that $v(j(E'))=12v(a_1)$ by applying Lemma \ref{le.supersing-va1-v2} (b) to $E'/K$. Thus $v(j)=12v(a_1)$. Therefore $v(j)\le 4s_{L/K}-4$ $\Longleftrightarrow$ $12v(a_1)\le 4s_{L/K}-4$ $\Longleftrightarrow$ $s_{L/K}-3v(a_1)\ge 1$. So Claim \ref{cl.equi} is true.
	\end{proof}

	Assume that we are in case (a) where $v(j)\le 4s_{L/K}-4$. Then $s_{L/K}-3v(a_1)\ge 1$ by Claim \ref{cl.equi}. Let $t:=(s_{L/K}-2v(a_1)+1)/2$ in Equation \eqref{eq.minequi}. Note that $t$ is an integer because $s_{L/K}$ is odd, and $t\ge 2$. Then $v(Da_1^2\pi^{2t})=1$, and 
	$$v((a_6+Da_3^2)\pi^{6t})=-s_{L/K}+6t=2(s_{L/K}-3v(a_1))+3\ge 5.$$ 
	
	Our equation \eqref{eq.minequi} satisfies the conditions of Step $7$ of Tate's algorithm (\cite{SilvermanAd} page 366). Indeed, $v(a_1\pi^t)\ge 1$, $v(Da_1^2\pi^{2t})=1$, $v(a_3\pi^{3t})\ge 3$, $v(a_4\pi^{4t})\ge 4$, and $v((a_6+Da_3^2)\pi^{6t})\ge 5$. The polynomial $P(T)$ considered in the algorithm has one simple root and one double root in $k$ because $v(Da_1^2\pi^{2t})=1$, $v(a_4\pi^{4t})> 2$, and $v((a_6+Da_3^2)\pi^{6t})> 3$. 
	
	In the $n$-th step of the subprocedure of Step $7$, given a general equation $y^2+A_1 xy+A_3 y=x^3+A_2 x^2+A_4 x+a_6$, we need to consider the polynomial $(A_2/\pi)x^2+(A_4/\pi^{(n+1)/2})x+A_6/\pi^n$ if $n$ is odd, and $y^2+(A_3/\pi^{n/2})y-A_6/\pi^n$ if $n$ is even. If the polynomial has a repeated root $0$ modulo $(\pi)$, the algorithm continues to Step $n+1$. If the polynomial has a repeated nonzero root modulo $(\pi)$, make a translation to the Weierstrass equation so that the repeated root modulo $(\pi)$ becomes $0$ and continue. The algorithm stops at Step $n$ when the polynomial has distinct roots modulo $(\pi)$, and the reduction is then $\mathrm{I}^{\ast}_{n-3}$.

	If $n<-s_{L/K}+6t$, then
	{\allowdisplaybreaks
		\begin{align*}
			v(a_3\pi^{3t})-n/2&\ge\frac{s_{L/K}}{2}>0,\\
			v(a_4 \pi^{4t})-(n+1)/2&\ge 4t-(n+1)/2=(3t-n/2)+(t-1/2)>0,\\
			v((a_6+Da_3^2)\pi^{6t})-n&>0.
		\end{align*}
	}
	So the polynomial in Step $n$ has a multiple root $0$ modulo $(\pi)$ and we can move to Step $n+1$.
	
	When we get to the subprocedure with $n=-s_{L/K}+6t$, we have that
	{\allowdisplaybreaks
		\begin{align*}
			v(Da_1^2 \pi^{2t})&=1,\\
			v(a_4 \pi^{4t})&>(n+1)/2,\\
			v((a_6+Da_3^2)\pi^{6t})&=n.
		\end{align*}
	}
	The relevant polynomial has a double root which is not $0$ modulo $\pi$ and we need to make a translation.
	The coefficient of $y$ in the Weierstrass equation after the translation is $0$ (see Remark \ref{rk.notation}). The coefficient of $x^2$, denoted by $a_2'$, in the Weierstrass equation after the translation still has valuation $1$. The coefficient of $x$ in the Weierstrass equation after the translation is $a_4\pi^{4t}+3(a_3a_1^{-1}\pi^{2t})^2$, which has valuation $4t-2v(a_1)$. In particular, $4t-2v(a_1)>(n+1)/2$. The new constant term $a_6'$ in the Weierstrass equation after the translation has valuation $6t-3v(a_1)$. In particular, $6t-3v(a_1)>n$. The algorithm continues.
	
	If $-s_{L/K}+6t<n<6t-3v(a_1)$, then the valuation of the coefficient of $x$ (resp. coefficient of $y$, constant term) in the Weierstrass equation is greater than $(n+1)/2$ (resp. $n/2$, $n$). So the algorithm continues without translation. 
	
	When $n=6t-3v(a_1)$, the valuation of the coefficient of $x$ (resp. coefficient of $y$) in the Weierstrass equation is greater than $(n+1)/2$ (resp. $n/2$) and the valuation of the constant term equals $n$. We need to work with the cases where $6t-3v(a_1)$ is even and odd separately. 
	
	{\bf Assume that $6t-3v(a_1)$ is odd}. Choose $b\in\mathcal{O}_K^*$ such that $v(a_2'\pi^{-1}b^2+a_6'\pi^{-6t+3v(a_1)})>0$. Make the translation $y=Y+b\pi^{(6t-3v(a_1)-1)/2}$. The coefficient of $y$ after the translation has valuation $(8t-v(a_1)-1)/2$. The coefficient of $x$ after the translation still has valuation $4t-2v(a_1)$. The constant term after the translation has valuation at least $6t-3v(a_1)+1$. The algorithm continues.

	Assume that $6t-3v(a_1)<n< 8t-4v(a_1)-1$. If $n$ is less than the valuation of the constant term in the Weierstrass equation, then the valuation of the coefficient of $x$ (resp. coefficient of $y$) is greater than $(n+1)/2$ (resp. $n/2$). So the algorithm continues without translation. If $n$ equals the valuation of the constant term, then make a translation so that the valuation of the new constant term becomes strictly bigger and the valuations of the coefficients of $x$ and $y$ are left unchanged. We claim that this is always possible.

	Assume that $6t-3v(a_1)<n_0< 8t-4v(a_1)-1$ and $n_0$ equals the valuation of the constant term. Let $A_1,...,A_6$ be the coefficients of the Weierstrass equation we get in Step $6t-3v(a_1)$. We will make a translation and use $A_1',...,A_6'$ to denote the coefficients of the Weierstrass equation after the translation. If $n_0$ is odd, let $b\in \mathcal{O}_K$ be such that $v(b^2+\frac{A_6}{A_2}\pi^{1-n_0})>0$ and let $r:=b\pi^{(n_0-1)/2}$. Make the translation $x=X+r$. Then $A_3'=A_3+rA_1$, $A_4'=A_4+3r^2$, and $A_6'=A_6+r A_4+r^2 A_2+r^3$ (see \cite{SilvermanArith} page 45). So $v(A_3')=v(A_3)$, $v(A_4')=v(A_4)$, and $v(A_6')>v(A_6)$ because $v(rA_4)>v(A_6)$ and $v(r^3)>v(A_6)$. If $n_0$ is even, let $b\in\mathcal{O}_K$ be such that $v(b^2-\frac{A_6}{\pi^{n_0}})>0$ and let $\overline{r}:=b\pi^{n_0}$. Make the translation $y=Y+\overline{r}$. Then $A_3'=A_3$, $A_4'=A_4-\overline{r} A_1$, and $A_6'=A_6-\overline{r} A_3-\overline{r}^2$ (see \cite{SilvermanArith} page 45). So $v(A_3')=v(A_3)$, $v(A_4')=v(A_4)$, and $v(A_6')>v(A_6)$ because $v(\overline{r}A_3)>v(A_6)$. So the translation can be done and the claim is true. Because the valuations of the coefficients of $x$ and $y$ are greater than $(n_0+1)/2$ and $n_0/2$, respectively, the algorithm continues.

	When $n=8t-4v(a_1)-1$, we consider the polynomial $g(x):=(A_2/\pi)x^2+(A_4/\pi^{(n+1)/2})x+A_6/\pi^n$ where $A_2$ (resp. $A_4, A_6$) is the coefficient of $x^2$ (resp. coefficient of $x$, constant term) after all the required translations of Equation \eqref{eq.minequi} in the subprocedure until the Step $n=8t-4v(a_1)-1$. We have $v(A_4)=4t-2v(a_1)=\frac{n+1}{2}$. It follows that $g(x)$ has distinct roots modulo $\pi$ as the coefficient of $x$ is a unit. Hence the reduction type is $\mathrm{I}_{n-3}^*$, with $n-3=8t-4v(a_1)-4=-12v(a_1)+4s_{L/K}=4s_{L/K}-v(j)$. 
	
	{\bf Assume that $6t-3v(a_1)$ is even}. This case is very similar to the case where $6t-3v(a_1)$ is odd, and is left to the reader.
	\vspace{3mm}

	Assume now that we are in case (b) where $v(j)> 4s_{L/K}-4$. Then $3v(a_1)-s_{L/K}> -1$ by Claim \ref{cl.equi}. Let $t:=\frac{f+s_{L/K}}{6}$ in Equation \eqref{eq.minequi}. Because $f/6$ is the fractional part of $\frac{2s_{L/K}+3}{6}$ and $s_{L/K}$ is odd, it follows that $t=\frac{3(s_{L/K}+1)}{6}-\frac{2s_{L/K}+3-f}{6}$ is an integer. Because $t\ge 1$, it is clear that $v(a_1\pi^{t})\ge 2,v(a_3\pi^{3{t}})\ge 3$, and $v(a_4 \pi^{4t})\ge 4$. Moreover,
	\[v(Da_1^2 \pi^{2t})=\frac{f+2(3v(a_1)-s_{L/K})}{3}\ge \frac{f}{3},\]
	and	
	\[v((a_6+Da_3^2)\pi^{6t})=-s_{L/K}+6t=f.\]

	If $f=1$, then Equation \eqref{eq.minequi} satisfies the conditions of Step $3$ in Tate's algorithm. Since $\pi^2\nmid (a_6+Da_3^2)\pi^{6t}$, the algorithm stops and $E/K$ has reduction type $\mathrm{II}$. If $f=3$, then Equation \eqref{eq.minequi} satisfies the conditions of Step $6$. Because the valuation of the discriminant of the polynomial $P(T)$ in the algorithm is $0$, the algorithm stops and $E/K$ has reduction type $\mathrm{I}^{\ast}_0$. If $f=5$, then Equation \eqref{eq.minequi} satisfies the conditions of Step $10$. Because $\pi^6\nmid (a_6+Da_3^2)\pi^{6t}$, the algorithm stops and $E/K$ has reduction type $\mathrm{II}^{\ast}$. This completes the proof of Theorem \ref{thm.supred} in the case of equicharacteristic $2$.

	\medskip
	
	{\it Now assume that $K$ is of mixed characteristic $2$}. Assume that $L:=K(\sqrt{D})$ with $D\in K$. Without loss of generality, we can assume that $v(D)=0$ or $1$ (Lemma \ref{le.sl_k} (b)). By a change of variables to Equation \eqref{eq.completeWE} over $K$, we can represent $E'/K$ by the equation 
	\[y^2=x^3+b_2x^2+8b_4x+16b_6,\]
	where $b_2=a_1^2+4a_2$, $b_4=2a_4+a_1a_3$, and $b_6=a_3^2+4a_6$. Then $E/K$ can be given by 
	\begin{equation}\label{eq.tw}
		y^2=x^3+Db_2x^2+8D^2b_4x+16D^3b_6.
	\end{equation}

	\textbf{The case $v(D)=1$.} In cases (a) and (b) below, we will choose a certain $t$ and consider the following dilation of Equation \eqref{eq.tw}:

	\begin{equation}\label{eq.minvd1}
		y^2=x^3+\pi^{2t} Db_2x^2+8\pi^{4t}D^2b_4x+16\pi^{6t}D^3b_6.
	\end{equation}
	\begin{claim}\label{cl.vd1}
		$v(j)\le 4s_{L/K}-4$ if and only if $2v(2)-3v(a_1)\ge 1$. 
	\end{claim}
	\begin{proof}[Proof of Claim \ref{cl.vd1}]
		Recall that $s_{L/K}=2v(2)$ (see Lemma \ref{le.sl_k} (b)). Assume that $v(j)\le 4s_{L/K}-4$. By Proposition \ref{prop.minext-twist} (b) and Lemma \ref{le.supersing-va1-v2} (b), we know that $v(j)=12v(a_1)$. So $2v(2)-3v(a_1)\ge 1$. Assume now the condition $2v(2)-3v(a_1)\ge 1$. It implies that $v(a_1)<v(2)$ and $12v(a_1)\le 4s_{L/K}-4$. Then $v(j)=3v(c_4)=3v(a_1^4+8a_1^2a_2+16a_2^2-48a_4-24a_1a_3)=12v(a_1)$ where $c_4$ is the standard invariant associated to \eqref{eq.completeWE}. Hence $v(j)\le 4s_{L/K}-4$. So Claim \ref{cl.vd1} is true.
	\end{proof}

	Assume that we are in case (a) where $v(j)\le 4s_{L/K}-4$. Then $2v(2)-3v(a_1)\ge 1$ by the claim \ref{cl.vd1}. Let $t:=-v(a_1)$ in Equation \eqref{eq.minvd1}. Then 
	{\allowdisplaybreaks
		\begin{align*}
			v(\pi^{2t} Db_2)&=2t+1+2v(a_1)=1,\\
			v(8\pi^{4t}D^2b_4)&=v(2)+2+(2v(2)-3v(a_1))\ge 4,\\
			v(16\pi^{6t}D^3b_6)&=2(2v(2)-3v(a_1))+3\ge 5.
	\end{align*}	}
	\noindent Our equation \eqref{eq.tw} satisfies the conditions of Step $7$ in Tate's algorithm (\cite{SilvermanAd} page 366). 
	
	If $n<4v(2)-6v(a_1)+3$, then
	{\allowdisplaybreaks
		\begin{align*}
			v(8\pi^{4t}D^2b_4)-(n+1)/2&=3v(2)-3v(a_1)+2-(n+1)/2\\
			&>3v(2)-3v(a_1)+2-\frac{4v(2)-6v(a_1)+4}{2}=v(2)>0,
		\end{align*}
	}
	and $v(16\pi^{6t}D^3b_6)-n>0$. So the polynomial in Step $n$ has multiple root $0$ modulo $(\pi)$ and we can move to Step $n+1$. 
	
	When we get to the subprocedure with $n=4v(2)-6v(a_1)+3\ge 5$, we have that
	{\allowdisplaybreaks
		\begin{align*}
			v(\pi^{2t} Db_2)&=1\\
			v(8\pi^{4t}D^2b_4)&>(n+1)/2,\\
			v((a_6+Da_3^2)\pi^{6t})&=n.
		\end{align*}
	}
	The relevant polynomial has a double root which is not $0$ modulo $\pi$ and we need to make a translation. Make the change of variable $x=X-4\pi^{2t}Da_1^{-2}(2a_4+a_1a_3)$. The coefficient of $x$ of the Weierstrass equation after the translation is $48\pi^{4t}D^2a_1^{-4}(2a_4+a_1a_3)^2,$ which has valuation $4v(2)-6v(a_1)+2$ and $4v(2)-6v(a_1)+2=n-1>\frac{n+1}{2}$. The new constant term $a_6'$ has valuation $6v(2)-9v(a_1)+3$ and $6v(2)-9v(a_1)+3=n+2v(2)-3v(a_1)>n$. The algorithm continues.
	
	If $4v(2)-6v(a_1)+3<n<6v(2)-9v(a_1)+3$, then the valuation of the coefficient of $x$ (resp. coefficient of $y$, constant term) is greater than $(n+1)/2$ (resp. $n/2$, $n$). So the algorithm continues and no translation is needed.

	When $n=6v(2)-9v(a_1)+3$, the valuation of the coefficient of $x$ (resp. coefficient of $y$) in the Weierstrass equation is greater than $(n+1)/2$ (resp. $n/2$) and the valuation of the constant term equals $n$. We need to deal with the cases with $6v(2)-9v(a_1)+3$ even and odd separately. 
	
	{\it Assume that $6v(2)-9v(a_1)+3$ is even}. Choose $c\in\mathcal{O}_K^*$ such that $\pi\mid c^2-a_6'\pi^{-(6v(2)-9v(a_1)+3)}$. Make the translation $y=Y+c\pi^{(6v(2)-9v(a_1)+3)/2}$. The new coefficient of $y$ is $2c\pi^{(6v(2)-9v(a_1)+3)/2}$, which has valuation $(8v(2)-9v(a_1)+3)/2$. The coefficient of $x$ still has valuation $4v(2)-6v(a_1)+2$. The constant term has valuation at least $6v(2)-9v(a_1)+4$. The algorithm continues.

	Assume that $6v(2)-9v(a_1)+3<n<8v(2)-12v(a_1)+3$. If $n$ is less than the valuation of the constant term in the Weierstrass equation, then the valuation of the coefficient of $x$ (resp. coefficient of $y$) is greater than $(n+1)/2$ (resp. $n/2$). So the algorithm continues without translation. If $n$ equals the valuation of the constant term, then make a translation so that the valuation of the new constant term becomes strictly bigger and the valuations of the coefficients of $x$ and $y$ are left unchanged. We claim that this is always possible.

	Assume that $6v(2)-9v(a_1)+3<n_0<8v(2)-12v(a_1)+3$ and $n_0$ equals the valuation of the constant term. Let $A_1,...,A_6$ be the coefficients of the Weierstrass equation we get in Step $6v(2)-9v(a_1)+3$. We will make a translation and use $A_1',...,A_6'$ to denote the coefficients of the Weierstrass equation after the translation. If $n_0$ is odd, let $b\in \mathcal{O}_K$ be such that $v(b^2+\frac{A_6}{A_2}\pi^{1-n_0})>0$ and let $r:=b\pi^{(n_0-1)/2}$. Make the translation $x=X+r$. Then $A_3'=A_3+rA_1$, $A_4'=A_4+2rA_2+3r^2$, and $A_6'=A_6+r A_4+r^2 A_2+r^3$ (see \cite{SilvermanArith} page 45). So $v(A_3')=v(A_3)$, $v(A_4')=v(A_4)$, and $v(A_6')>v(A_6)$ because $v(rA_4)>v(A_6)$ and $v(r^3)>v(A_6)$. If $n_0$ is even, let $b\in\mathcal{O}_K$ be such that $v(b^2-\frac{A_6}{\pi^{n_0}})>0$ and let $\overline{r}:=b\pi^{n_0}$. Make the translation $y=Y+\overline{r}$. Then $A_3'=A_3+2\overline{r}$, $A_4'=A_4-\overline{r} A_1$, and $A_6'=A_6-\overline{r} A_3-\overline{r}^2$ (see \cite{SilvermanArith} page 45). So $v(A_3')=v(A_3)$, $v(A_4')=v(A_4)$, and $v(A_6')>v(A_6)$ because $v(\overline{r}A_3)>v(A_6)$. So the translation can be done and the claim is true. Because the valuations of the coefficients of $x$ and $y$ are greater than $(n_0+1)/2$ and $n_0/2$, respectively, the algorithm continues.

	When $n=8v(2)-12v(a_1)+3$, consider the polynomial $g(x):=(A_2/\pi)x^2+(A_4/\pi^{(n+1)/2})x+A_6/\pi^n$ where $A_2$ (resp. $A_4, A_6$) is the coefficient of $x^2$ (resp. coefficient of $x$, constant term) after all the required translations of Equation \eqref{eq.minvd1} in the subprocedure until the Step $n=8v(2)-12v(a_1)+3$. We have $v(A_4)=4v(2)-6v(a_1)+2$. It follows that $g(x)$ has distinct roots modulo $\pi$ as the coefficient of $x$ is a unit. Hence the reduction type is $\mathrm{I}_{n-3}^*$, with $n-3=8v(2)-12v(a_1)=4s_{L/K}-v(j)$.

	{\it Assume that $6v(2)-9v(a_1)+3$ is odd}. This case is similar to the case where $6v(2)-9v(a_1)+3$ is odd, and is left to the reader.
	
	\medskip
	
	Assume that we are in case (b) where $v(j)> 4s_{L/K}-4$. Then $3v(a_1)-2v(2)\ge 0$ by Claim \ref{cl.vd1}. Let $t:=\frac{f-4v(2)-3}{6}$ in Equation \eqref{eq.minvd1}. Recall that $\frac{f}{6}$ is the fractional part of $\frac{2s_{L/K}+3}{6}$ and that $s_{L/K}=2v(2)$ (Lemma \ref{le.sl_k} (b)). So $t$ is an integer. Then 
	{\allowdisplaybreaks
		\begin{align*}
			v(\pi^{2t} Db_2)&=\frac{f+2(3v(a_1)-2v(2))}{3}\ge \frac{f}{3},\\
			v(8\pi^{4t}D^2b_4)&=\frac{2f+v(2)+3v(b_4)}{3}\ge \frac{2f+4}{3},\\
			v(16\pi^{6t}D^3b_6)&=4v(2)+3+6t=f.
		\end{align*}
	}

	If $f=1$, then Equation \eqref{eq.minvd1} satisfies the conditions of Step $3$ in Tate's algorithm. Indeed, $v(\pi^{2t} Db_2)\ge 1$, $v(8\pi^{4t}D^2b_4)\ge 2$, and $v(16\pi^{6t}D^3b_6)=1$. Because $\pi^2\nmid 16\pi^{6t}D^3b_6$, $E/K$ has reduction type $\mathrm{II}$. 
	
	If $f=3$, then Equation \eqref{eq.minvd1} satisfies the conditions of Step $6$. Indeed, $v(\pi^{2t} Db_2)\ge 1$, $v(8\pi^{4t}D^2b_4)\ge 4$, and $v(16\pi^{6t}D^3b_6)=3$. The valuation of the discriminant of the polynomial $P(T)$ in the algorithm is $0$. So $E/K$ has reduction type $\mathrm{I}^{\ast}_0$. 
	
	If $f=5$, then Equation \eqref{eq.minvd1} satisfies the conditions of Step $10$. Indeed, $v(\pi^{2t} Db_2)\ge 2$, $v(8\pi^{4t}D^2b_4)\ge 5$, and $v(16\pi^{6t}D^3b_6)=5$. Since $\pi^6\nmid 16\pi^{6t}D^3b_6$, $E/K$ has reduction type $\mathrm{II}^{\ast}$. This finishes the proof of the case $v(D)=1$.

	\medskip

	\textbf{The case $v(D)=0$.} 
	By Lemma \ref{le.sl_k} (b), $L/K$ is given by an Eisenstein polynomial $z^2+az+b\in K[z]$ with $v(a)\ge 1$, $v(b)=1$ and $a^2-4b=Dc^2$ for some $c\in K$ with $v(c)=v(a)\le v(2)$. Let $m:=ac^{-1}$. Since $D-m^2=-4b/c^2$, we have $v(D-m^2)=2v(2)-2v(a)+1$. To make the valuation of the coefficient of $x^2$ odd in a new equation, consider the change of variable $y=Y+a_1mx+4Da_3m$ to \eqref{eq.tw}. We get the following equation
	{\allowdisplaybreaks
		\begin{equation}\label{eq.twvD0}
			\aligned
			&y^2+2a_1mxy+8Da_3my
			=x^3+(a_1^2(D-m^2))x^2\\
			&\pushright{\qquad\qquad\qquad\qquad\qquad\qquad+8D(a_1a_3(D-m^2)+2Da_4)x+16D^2(a_3^2(D-m^2)+4Da_6)}
			\endaligned
		\end{equation}
	}
	(see Remark \ref{rk.notation}). In cases (a) and (b) below, we will choose a certain $t$ and consider the following dilation of Equation \eqref{eq.twvD0}: 
	{\allowdisplaybreaks
		\begin{equation}\label{eq.minvd0}
			\aligned
			&y^2+2\pi^t a_1mxy+8\pi^{3t} Da_3my
			=x^3+\pi^{2t} (a_1^2(D-m^2))x^2\\
			&\qquad\qquad\qquad\qquad+8\pi^{4t} D(a_1a_3(D-m^2)+2Da_4)x+16\pi^{6t} D^2(a_3^2(D-m^2)+4Da_6).
			\endaligned
		\end{equation}
	}
	
	\begin{claim}\label{cl.vd0}
		$v(j)\le 4s_{L/K}-4$ if and only if $2v(a)-3v(a_1)\ge 2$.
	\end{claim}
	
	\begin{proof}[Proof of Claim \ref{cl.vd1}]
		Recall that $s_{L/K}=2v(a)-1$ (see Lemma \ref{le.sl_k}). Assume that $v(j)\le 4s_{L/K}-4$. By Proposition \ref{prop.minext-twist} (b) and Lemma \ref{le.supersing-va1-v2} (b), we get that $v(j)=12v(a_1)$. So $2v(a)-3v(a_1)\ge 2$. Conversely, assume that $2v(a)-3v(a_1)\ge 2$. Then $v(a_1)<v(a)\le v(2)$. It follows that $v(j)=3v(c_4)=3v(a_1^4+8a_1^2a_2+16a_2^2-48a_4-24a_1a_3)=12v(a_1)$. Hence $v(j)=12v(a_1)\le8v(a)-8= 4s_{L/K}-4$. So Claim \ref{cl.vd0} holds.
	\end{proof}

	Assume that we are in case (a) where $v(j)\le 4s_{L/K}-4$. Then $2v(a)-3v(a_1)\ge 2$ by Claim \ref{cl.vd0}. Let $t:=-v(a_1)-v(2)+v(a)$ in Equation \eqref{eq.minvd0}. Then $v(2\pi^{t} a_1m)=v(a)\ge 1$, $v(\pi^{2t} (a_1^2(D-m^2)))=2t+2v(a_1)+2v(2)-2v(a)+1=1$, and $v(8\pi^{3t} Da_3m)=3v(a)-3v(a_1)\ge 3$. If $v(a_1)+2v(2)-2v(a)+1<v(2)$, then \[v(8\pi^{4t} D(a_1a_3(D-m^2)+2Da_4))=v(2)+2v(a)-3v(a_1)+1\ge 4.\]
	If $v(a_1)+2v(2)-2v(a)+1\ge v(2)$, then \[v(8\pi^{4t} D(a_1a_3(D-m^2)+2Da_4))\ge 2(2v(a)-3v(a_1))+2v(a_1)\ge 6.\] 
	Finally, \[v(16\pi^{6t} D^2(a_3^2(D-m^2)+4Da_6))=2(2v(a)-3v(a_1))+1\ge 5.\] 
	It follows that Equation \eqref{eq.minvd0} satisfies the conditions of Step $7$ in Tate's algorithm (\cite{SilvermanAd} page 366). If $n<4v(a)-6v(a_1)+1$, then 
	{\allowdisplaybreaks
		\begin{align*}
			v(8\pi^{3t} Da_3m)-n/2>3v(a)-&3v(a_1)-\frac{4v(a)-6v(a_1)+1}{2}=\frac{2v(a)-1}{2}>0,\\
			v(16\pi^{6t} D^2(a_3^2(D-m^2)&+4Da_6))-n>0.
		\end{align*}
	}
	If $v(a_1)+2v(2)-2v(a)+1<v(2)$, then 
	{\allowdisplaybreaks
		\begin{align*}
			v(8\pi^{4t} D(a_1a_3(D-m^2)+2Da_4))-(n+1)/2=&v(2)+2v(a)-3v(a_1)+1-(n+1)/2\\
			>&v(2)+2v(a)-3v(a_1)+1-\frac{4v(a)-6v(a_1)+2}{2}\\
			=&\frac{2v(2)-2}{2}\ge 0
		\end{align*}
	}
	If $v(a_1)+2v(2)-2v(a)+1\ge v(2)$, then 
	{\allowdisplaybreaks
		\begin{align*}
			v(8\pi^{4t} D(a_1a_3(D-m^2)+2Da_4))-(n+1)/2&\ge4v(a)-4v(a_1)-(n+1)/2\\
			&>4v(a)-4v(a_1)-\frac{4v(a)-6v(a_1)+2}{2}\\
			&=\frac{4v(a)-6v(a_1)}{2}+2v(a_1)-1>0
		\end{align*}
	}
	So the polynomial in Step $n$ has multiple root $0$ modulo $(\pi)$ and we can move to Step $n+1$. When we get to the subprocedure with $n=4v(a)-6v(a_1)+1\ge 5$, we have that 
	{\allowdisplaybreaks
		\begin{align*}
			v(\pi^{2t} (a_1^2(D-m^2)))&=1\\
			v(8\pi^{4t} D(a_1a_3(D-m^2)+2Da_4))&>(n+1)/2,\\
			v(16\pi^{6t} D^2(a_3^2(D-m^2)+4Da_6))&=n.
		\end{align*}
	}
	The relevant polynomial has a double root which is not $0$ modulo $\pi$ and we need to make a translation. Make the change of variable $x=X-4\pi^{2t}Da_3a_1^{-1}$. The coefficient of $y$ of the Weierstrass equation after the translation is $0$. The coefficient of $x$ of the Weierstrass equation after the translation has valuation $4v(a)-6v(a_1)=n-1>\frac{n+1}{2}$. The constant term $a_6'$ has valuation $6v(a)-9v(a_1)=\frac{3}{2}(n-1)>n$. The algorithm continues.

	If $4v(a)-6v(a_1)+1<n<6v(a)-9v(a_1)$, then the valuation of the coefficient of $x$ (resp. coefficient of $y$, constant term) is greater than $(n+1)/2$ (resp. $n/2$, $n$). So the algorithm continues and no translation is needed. 
	
	When $n=6v(a)-9v(a_1)$, the valuation of the coefficient of $x$ (resp. coefficient of $y$) in the Weierstrass equation is greater than $(n+1)/2$ (resp. $n/2$) and the valuation of the constant term equals $n$. We need to consider the cases with $6v(a)-9v(a_1)$ even and odd separately. 
	
	{\it Assume that $6v(a)-9v(a_1)$ is even}. Choose $d\in\mathcal{O}_K^*$ such that $v(d^2-a_6'\pi^{-(6v(a)-9v(a_1))})>0$. Make the translation $y=Y+d\pi^{(6v(a)-9v(a_1))/2}$. The new coefficient of $y$ is $2d\pi^{(6v(a)-9v(a_1))/2}$, which has valuation $(2v(2)+6v(a)-9v(a_1))/2$. The coefficient of $x$ still has valuation $4v(a)-6v(a_1)$. The constant term has valuation at least $6v(a)-9v(a_1)+1$. The algorithm continues.
	
	Assume that $6v(a)-9v(a_1)<n<8v(a)-12v(a_1)-1$. If $n$ is less than the valuation of the constant term in the Weierstrass equation, then the valuation of the coefficient of $x$ (resp. coefficient of $y$) is greater than $(n+1)/2$ (resp. $n/2$). So the algorithm continues without translation. If $n$ equals the valuation of the constant term, then make a translation so that the valuation of the new constant term becomes strictly bigger and the valuations of the coefficients of $x$ and $y$ are left unchanged. As in Step $n$ when $6v(2)-9v(a_1)+3<n<8v(2)-12v(a_1)+3$ in the case of $v(D)=1$, this is always possible. Because the valuations of the coefficients of $x$ and $y$ are greater than $(n+1)/2$ and $n/2$, respectively, the algorithm continues.
	
	When $n=8v(a)-12v(a_1)-1$, we consider the appropriate polynomial of the form $g(x):=(A_2/\pi)x^2+(A_4/\pi^{(n+1)/2})x+A_6/\pi^n$ where $A_2$ (resp. $A_4, A_6$) is the coefficient of $x^2$ (resp. coefficient of $x$, constant term) after all the required translations of Equation \eqref{eq.minvd0} in the subprocedure until the Step $n=8v(a)-12v(a_1)-1$. We have $v(A_4)=4v(a)-6v(a_1)$. It follows that $g(x)$ has distinct roots modulo $\pi$ as the coefficient of $x$ is a unit. This polynomial has distinct roots modulo $\pi$ as the coefficient of $x$ is a unit. Hence the reduction type is $\mathrm{I}_{n-3}^*$, with $n-3=8v(a)-12v(a_1)-4=4s_{L/K}-v(j)$. 
	
	{\it Assume that $6v(a)-9v(a_1)$ is odd}. This case is similar to the case where $6v(a)-9v(a_1)$ is even, and is left to the reader.
	
	\medskip
	
	Assume that we are in case (b) where $v(j)> 4s_{L/K}-4$. Then $3v(a_1)-2v(a)>-2$ by Claim \ref{cl.vd0}. Let $t:=\frac{2v(a)-1+f}{6}-v(2)$ in Equation \eqref{eq.minvd0}, which is an integer because $s_{L/K}=2v(a)-1$, $\frac{f}{6}$ is the fractional part of $\frac{2s_{L/K}+3}{6}$, and $\frac{2v(a)-1+f}{6}=v(a)-\frac{4v(a)+1-f}{6}$. Then
	{\allowdisplaybreaks
		\begin{align*}
			v(2\pi^{t} a_1m)&=\frac{f+6v(a_1)+2v(a)-1}{6}\ge \frac{f+7}{6},\\
			v(\pi^{2t} (a_1^2(D-m^2)))&=\frac{f+2(3v(a_1)-2v(a))+2}{3}\ge \frac{f}{3},\\
			v(8\pi^{3t} Da_3m)&=\frac{f+2v(a)-1}{2}\ge \frac{f+1}{2},\\
			v(16\pi^{6t} D^2(a_3^2(D-m^2)+4Da_6))&=6v(2)+6t-2v(a)+1=f.
		\end{align*}
	}
	If $v(a_1)+2v(2)-2v(a)+1<v(2)$, then 
	{\allowdisplaybreaks
		\begin{align*}
			v(8\pi^{4t} D(a_1a_3(D-m^2)+2Da_4))&=\frac{2f+3v(2)+3v(a_1)-2v(a)+1}{3}\\
			&\ge\frac{2f+3}{3}. 
		\end{align*}
	}
	If $v(a_1)+2v(2)-2v(a)+1\ge v(2)$, then 
	{\allowdisplaybreaks
		\begin{align*}
			v(8\pi^{4t} D(a_1a_3(D-m^2)+2Da_4))&\ge \frac{2f+4v(a)-2}{3}\\
			&\ge \frac{2f+2}{3}.
		\end{align*}
	}

	If $f=1$, then Equation \eqref{eq.minvd0} satisfies the conditions of Step $3$ in Tate's algorithm. Indeed, $v(2\pi^{t} a_1m)\ge 2$, $v(\pi^{2t} (a_1^2(D-m^2)))\ge 1$, $v(8\pi^{3t} Da_3m)\ge 1$, $v(8\pi^{4t} D(a_1a_3(D-m^2)+2Da_4))\ge 2$, and $v(16\pi^{6t} D^2(a_3^2(D-m^2)+4Da_6))=1$. Because $\pi^2\nmid 16\pi^{6t} D^2(a_3^2(D-m^2)+4Da_6)$, $E/K$ has reduction type $\mathrm{II}$. 
	
	If $f=3$, then Equation \eqref{eq.minvd0} satisfies the conditions of Step $6$. Indeed, $v(2\pi^{t} a_1m)\ge 2$, $v(\pi^{2t} (a_1^2(D-m^2)))\ge 1$, $v(8\pi^{3t} Da_3m)\ge 2$, $v(8\pi^{4t} D(a_1a_3(D-m^2)+2Da_4))\ge 3$, and $v(16\pi^{6t} D^2(a_3^2(D-m^2)+4Da_6))=3$. The valuation of the discriminant of the polynomial $P(T)$ in the algorithm is $0$. So $E/K$ has reduction type $\mathrm{I}^{\ast}_0$. 
	
	If $f=5$, then Equation \eqref{eq.minvd0} satisfies the conditions of Step $10$. Indeed, $v(2\pi^{t} a_1m)\ge 2$, $v(\pi^{2t} (a_1^2(D-m^2)))\ge 2$, $v(8\pi^{3t} Da_3m)\ge 3$, $v(8\pi^{4t} D(a_1a_3(D-m^2)+2Da_4))\ge 4$, and $v(16\pi^{6t} D^2(a_3^2(D-m^2)+4Da_6))=5$. Since $\pi^6\nmid 16\pi^{6t} D^2(a_3^2(D-m^2)+4Da_6)$, $E/K$ has reduction type $\mathrm{II}^{\ast}$. This finishes the proof of the case $v(D)=0$. 
	
	Assume that we are in case (a), i.e. $v(j)\le 4s_{L/K}-4$. Then Proposition \ref{prop.minext-twist} (b) shows that $12\mid v(j)$. The proof of Theorem \ref{thm.supred} is now complete.
\end{proof}

\begin{remark}
	The mixed characteristic case of Theorem \ref{thm.supred} can be derived from \cite{Comalada}, Corollary 6, with some work.
\end{remark}

\begin{corollary}\label{cor.I_ms}
	Let $\mathcal{O}_K$ be as in \ref{em.intro}. Let $E/K$ be an elliptic curve with reduction type $\mathrm{I}^{\ast}_{n}$ for some $n\ge 1$. Assume that there exists a quadratic extension $L/K$ such that $E_L/L$ has semistable reduction. Then $n=4s_{L/K}-v(j(E))$.
\end{corollary}
\begin{proof}
	The statement follows from Theorem 2.8 in \cite{Lorenzini10}, Theorem 4.2 in \cite{Lorenzini13}, and Theorem \ref{thm.supred} above. Note that if $E_K/K$ has potentially multiplicative reduction or potentially good ordinary reduction, then a quadratic extension $L/K$ such that $E_L/L$ is semistable always exists.
\end{proof}

\begin{remark}
	Let $\mathcal{O}_K$ be as in \ref{em.intro}. If $E/K$ has reduction type $\mathrm{I}^{\ast}_{n}$ for some $n\ge 1$ and good ordinary or supersingular reduction after a quadratic extension, then $4\mid n$ by Corollary \ref{cor.I_ms} and Theorem \ref{thm.supred} (a). However, the divisibility is not necessarily true for $E/K$ with additive reduction and potentially multiplicative reduction. For instance, consider the elliptic curve $E/\mathbb{Q}$ with LMFDB label 432.g2. The elliptic curve has reduction type $\mathrm{I}^{\ast}_{5}$ and potentially multiplicative reduction at $(2)$.
\end{remark}

\begin{example}
	
	Let $\mathcal{O}_K$ be as in \ref{em.intro}. Let $E/K$ be an elliptic curve with additive reduction and potentially good supersingular reduction after quadratic extension. Theorem \ref{thm.supred} (a) shows that if $E/K$ has reduction type $\mathrm{I}^{\ast}_{n}$, then $4\mid n$. The following two examples illustrate that there indeed exist elliptic curves with reduction type $\mathrm{I}^{\ast}_{n}$ with $4\nmid n$ and potentially good supersingular reduction.

	(a) Let $E_1/K_1$ be the elliptic curve with LMFDB label 2.2.8.1-128.1-a1 \cite{LMFDB} where $K_1=\mathbb{Q}(\sqrt{2})$. Then $j(E_1)=2^7$. So $E_1/K_1$ has potentially good supersingular reduction at the place $(\sqrt{2})$. The curve $E_1/K_1$ has reduction type $\mathrm{I}^{\ast}_5$ at $(\sqrt{2})$.

	(b) Let $E_2/K_2$ be the elliptic curve with LMFDB label 2.0.4.1-4096.1-a2 where $K_2=\mathbb{Q}(\sqrt{-1})$. Then $j(E_2)=2^6 5^3$. So $E_2/K_2$ has potentially good supersingular reduction at the only prime over $(2)$. The curve $E_2/K_2$ has reduction type $\mathrm{I}^{\ast}_2$ at the prime over $(2)$.

	Theorem \ref{thm.supred} (a) shows that the base change $E_{1,L}/L$ (resp. $E_{2,L}/L$) is not semistable at any place over $(2)$ for any quadratic extension $L/K_1$ (resp. $L/K_2$).
	
\end{example}

\begin{lemma}\label{le.jinv}
	Let $\mathcal{O}_K$ be as in \ref{em.intro} and let $0<u<v(2)$. Then there exists an elliptic curve $E/K$ with good supersingular reduction such that $v(j(E))=12u$.
\end{lemma}

\begin{proof}

	Let $a_1=\pi_K^u$ and $a_3\in\mathcal{O}_K^*$. Let $E/K$ be the elliptic curve given by
	\[y^2+a_1 xy+a_3y=x^3.\]

	By Lemma \ref{le.supersing-va1-v2}, $E/K$ has good supersingular reduction. Since $v(1728)=6v(2), v(13824)=9v(2)$, and $v(a_1)<v(2)$, we obtain that
	\[v(j)=v(\frac{a_1^{12}-72a_1^9a_3+1728a_1^6a_3^2-13824a_1^3a_3^3}{a_1^3a_3^3-27a_3^4})=12v(a_1)=12u.\]
\end{proof}

Note that $E/K$ in Lemma \ref{le.jinv} is endowed with a $3$-isogeny $E/K\rightarrow E_2/K$ over $K$ (see \cite{Hadano} sec. 1). It follows from Proposition \ref{prop.vj} that $v(j(E_2))=12u$.

\begin{corollary}\label{cor.Im}
	Let $\mathcal{O}_K$ be as in \ref{em.intro} with mixed characteristic $2$ and $v(2)\ge 2$. Let $E/K$ be an elliptic curve with additive reduction and good supersingular reduction after a quadratic extension $L/K$. Assume that $v(j(E))\le 4s_{L/K}-4$. Then the reduction type of $E/K$ can only be $\mathrm{I}^{\ast}_{4m}$ with $m$ satisfying (i) $1\le m\le 2v(2)-3$ and (ii) $m\ne 2v(2)-5$. 
	
	Conversely, let $m$ be an integer which satisfies (i) and (ii). Then there exists an elliptic curve $E/K$ with reduction type $\mathrm{I}^{\ast}_{4m}$ and good supersingular reduction after a quadratic extension.
\end{corollary}



\begin{proof}

	By Theorem \ref{thm.supred} (a), the reduction type of $E/K$ is $\mathrm{I}^{\ast}_{4m}$ with $m=(4s_{L/K}-v(j))/4$. Recall that Remark \ref{rk.gt12} shows that $v(j)\ge 12$. Therefore $m\le (-12+4s_{L/K})/4\le (-12+8v(2))/4=2v(2)-3$ by Lemma \ref{le.slkbound}. We claim that $m\ne 2v(2)-5$. Assume that $4s_{L/K}-v(j)=4(2v(2)-5)$. Then $s_{L/K}=2v(2)-5+v(j)/4\ge 2v(2)-2$ by Remark \ref{rk.gt12}. Hence $s_{L/K}=2v(2)-1$ or $s_{L/K}=2v(2)$ by Lemma \ref{le.slkbound}. Then $v(j)=16$ or $v(j)=20$. We get a contradiction in either case because $12\mid v(j)$ by Proposition \ref{prop.minext-twist} (b).

	Now we prove the converse direction. Let $m$ be an integer which satisfies (i) and (ii). If $m=2v(2)-3$, set $s:=m+3$. Assume that $1\le m< 2v(2)-5$ or $m=2v(2)-4$. If $m$ is odd, set $s:=m+6$. If $m$ is even, set $s:=m+3$.
	
	Let $u:=(s-m)/3$. If $v(2)=2$, then $m=1$ since $m$ satisfies (i). Because $m=2v(2)-3$, it follows that $s=m+3$ and $0<u=1<v(2)$. If $v(2)\ge 3$, then $0<u=(s-m)/3\le 2<v(2)$. So, we always have that $0<u<v(2)$.

	By Lemma \ref{le.jinv}, there exists an elliptic curve $E'/K$ with good supersingular reduction such that $v(j(E'))=12u$. Take $L$ to be a quadratic extension such that $s_{L/K}=s$ (see Lemma \ref{le.slkbound}). Let $E/K$ be the quadratic twist of $E'/K$ by $L/K$. Then $E_L/L$ has good supersingular reduction. By construction, $v(j(E))=4(s-m)\le 4s_{L/K}-4$. Note now that $4m=-12u+4s$. It follows that $E/K$ has reduction type $\mathrm{I}^{\ast}_{4m}$ by Theorem \ref{thm.supred} (a).
\end{proof}

\begin{corollary}\label{cor.v2_sm}
	Let $\mathcal{O}_K$ be as in \ref{em.intro} with mixed characteristic $2$. Let $E/K$ be an elliptic curve with additive reduction and good supersingular reduction after a quadratic extension $L/K$.
	\begin{enumerate}[label=(\alph*)]
		\item Assume that $v(2)=1$. Then the reduction type of $E/K$ can only be $\mathrm{II}$ or $\mathrm{II}^{\ast}$. 
		\item Assume that $v(2)=2$. Then the reduction type of $E/K$ can only be $\mathrm{I}^{\ast}_0,\mathrm{II}^{\ast}$ or $\mathrm{I}^{\ast}_4$.
		\item Assume that $v(2)=3$. Then the reduction type of $E/K$ can only be $\mathrm{II},\mathrm{I}^{\ast}_0,\mathrm{II}^{\ast},\mathrm{I}^{\ast}_8$ or $\mathrm{I}^{\ast}_{12}$.
	\end{enumerate}
	
	Conversely, assume that $v(2)=1\;(\text{resp. 2,3})$. Then for each reduction type $t$ in (a) (resp. (b),(c)), there exists an elliptic curve $E/K$ with good supersingular reduction after a quadratic extension and reduction type $t$. Moreover, if $v(2)\ge 4$, there exists such an elliptic curve with reduction type $\mathrm{II},\mathrm{I}^{\ast}_0$, or $\mathrm{II}^{\ast}$.
\end{corollary}

\begin{proof}
	Let $E'/K$ be the twist of $E/K$ by the quadratic extension $L/K$.
	By the proof of Theorem \ref{thm.supred}, we know that $E'/K$ has good supersingular reduction. Let $y^2+a_1 xy+a_3 y=x^3+a_2x^2+a_4x+a_6$ with $a_i\in\mathcal{O}_K$ be a minimal Weierstrass equation for $E'/K$.
	
	Assume that $v(2)=1$. We claim that $v(j)\ge 12$. If $v(j)> 12v(2)$, then $v(j)>12$ since $v(2)\ge 1$. If $v(j)\le 12v(2)$, then Proposition \ref{prop.minext-twist} (b) shows that $v(j(E))=12v(a_1)\ge 12$. Meanwhile, Lemma \ref{le.sl_k} shows that $s_{L/K}\le 2v(2)=2$. Hence $v(j(E))\ge 12> 4s_{L/K}-4$. By Lemma \ref{le.slkbound}, $s_{L/K}=1$ or $2$.
	If $s_{L/K}=1$, then $E/K$ has reduction type $\mathrm{II}^{\ast}$ and if $s_{L/K}=2$, then $E/K$ has reduction type $\mathrm{II}$ by Theorem \ref{thm.supred}.

	Assume that $v(2)=2$. If $v(j)\le 4s_{L/K}-4$, then the reduction type of $E/K$ can only be $\mathrm{I}^{\ast}_4$ by Corollary \ref{cor.Im}. Assume that $v(j)> 4s_{L/K}-4$. By Lemma \ref{le.slkbound}, $s_{L/K}=1,3$ or $4$. If $s=1$ or $4$, then $E/K$ has reduction type $\mathrm{II}^{\ast}$ and if $s=3$, then $E/K$ has reduction type $\mathrm{I}^{\ast}_0$ by Theorem \ref{thm.supred}.
	
	Assume that $v(2)=3$. If $v(j)\le 4s_{L/K}-4$, then the reduction type of $E/K$ can only be $\mathrm{I}^{\ast}_8$ or $\mathrm{I}^{\ast}_{12}$ by Corollary \ref{cor.Im}. Assume that $v(j)> 4s_{L/K}-4$, then the reduction type of $E/K$ can only be $\mathrm{II},\mathrm{I}^{\ast}_0$ or $\mathrm{II}^{\ast}$ by Theorem \ref{thm.supred}.

	We now show the converse direction. Let $a_3\in\mathcal{O}_K^*$ and let $E'/K$ to be the elliptic curve $y^2+a_3y=x^3$. Then $j(E')=0$.
	
	Assume that $v(2)=1$. Fix $s=1(\text{resp. } 2)$. There exists a quadratic extension $L/K$ such that $s_{L/K}=s$ by Lemma \ref{le.slkbound}. Let $E/K$ be the quadratic twist of $E'/K$ by the extension $L/K$. Then $E/K$ has reduction type $\mathrm{II}^{\ast}\,\,(\text{resp. } \mathrm{II})$.
	
	Assume that $v(2)=2$. Fix $s=1,3$ or $4$. There exists a quadratic extension $L/K$ such that $s_{L/K}=s$ by Lemma \ref{le.slkbound}. Let $E/K$ be the quadratic twist of $E'/K$ by the extension $L/K$. If $s=1$ or $4$, then $E/K$ has reduction type $\mathrm{II}^{\ast}$ and if $s=3$, then $E/K$ has reduction type $\mathrm{I}^{\ast}_0$ by Theorem \ref{thm.supred}.

	Assume that $v(2)\ge 3$. Fix $s=1(\text{resp. } 3,5)$. There exists a quadratic extension $L/K$ such that $s_{L/K}=s$ by Lemma \ref{le.slkbound}. Let $E/K$ be the quadratic twist of $E'/K$ by the extension $L/K$. Then $E/K$ has reduction type $\mathrm{II}^{\ast} \,\,(\text{resp. } \mathrm{I}^{\ast}_0, \mathrm{II})$ by Theorem \ref{thm.supred}.
	
	By Corollary \ref{cor.Im}, if $v(2)=2 (\text{ resp. 3})$, then there exists an elliptic curve $E/K$ with reduction type $\mathrm{I}^{\ast}_4$ (resp. $\mathrm{I}^{\ast}_8$ or $\mathrm{I}^{\ast}_{12}$) and good supersingular reduction after a quadratic extension. 
\end{proof}

\begin{lemma}
	Let $\mathcal{O}_K$ be as in \ref{em.intro} with equicharacteristic $2$. Then the following is true.
	
	\begin{enumerate}[label=(\alph*)]
		\item Let $m>0$ be an integer. There exists an elliptic curve over $K$ with reduction type $\mathrm{I}^{\ast}_{4m}$ and good supersingular reduction after a quadratic extension.
		\item There exists an elliptic curve over $K$ with reduction type $\mathrm{II}$ $(\text{resp. }\mathrm{I}^{\ast}_0$, $\mathrm{II}^{\ast}$) and good supersingular reduction after a quadratic extension. 
	\end{enumerate}
	
\end{lemma}

\begin{proof}
	(a) Let $s$ be an odd integer such that $m<s$ and $3\mid (s-m)$. Let $u:=4(s-m)$. Then $12\mid u$. By Lemma \ref{le.jinv}, there exists an elliptic curve $E'/K$ with good supersingular reduction such that $v(j(E'))=u$. By Lemma \ref{le.slkbound}, there exists a quadratic extension $L/K$ such that $s_{L/K}=s$. Let $E/K$ be the quadratic twist of $E'/K$ by the extension $L/K$. Then $v(j(E))\le 4s_{L/K}-4$ since $u\le 4s-4$. It follows that $E/K$ has reduction type $\mathrm{I}^{\ast}_{4m}$ by Theorem \ref{thm.supred} since $4s_{L/K}-v(j)=-u+4s=4s-(4s-4m)=4m$.
	
	(b) Let $a_3\in\mathcal{O}_K^*$. Take $E'/K$ to be the elliptic curve $y^2+a_3y=x^3$. Then $j(E')=0$. Let $s>0$ be an odd integer such that $2s+3\equiv 1 (\text{resp. 3,5}) \operatorname{mod} 6$. Then there exists a quadratic extension $L/K$ such that $s_{L/K}=s$ by Lemma \ref{le.slkbound}. Let $E/K$ be the quadratic twist of $E'/K$ by the extension $L/K$. Then $E/K$ has reduction type $\mathrm{II} \,\,(\text{resp. } \mathrm{I}^{\ast}_0, \mathrm{II}^{\ast})$.
\end{proof}

\section{Reduction type under isogeny}\label{sec.isog}

In this section, we study the relation between Kodaira types of     isogenous elliptic curves. The following lemma might be well-known, but we did not find a reference for it in the literature.

\begin{lemma}\label{le.supersing}
	Let $\mathcal{O}_K$ be a discrete valuation domain with field of fractions $K$. Let $E_1/K$ and $E_2/K$ be isogenous elliptic curves with good reduction. If $E_1/K$ has supersingular reduction, then $E_2/K$ has supersingular reduction.
\end{lemma}

\begin{proof}
	Let $k$ be the residue field of $\mathcal{O}_K$ and let $p$ be the characteristic of $k$. Let $\phi:E_1/K\rightarrow E_2/K$ be an isogeny over $K$ with dual isogeny $\psi:E_2/K\rightarrow E_1/K$. Let $d$ be the degree of $\phi$. Then $\psi\circ \phi$ is the multiplication by $d$ map on $E_1/K$. Let $\mathcal{E}_i/\mathcal{O}_K$ be the Néron model of $E_i/K$ for $i=1,2$. By the Néron mapping property, the morphisms $\phi$ (resp. $\psi$) extends uniquely to a morphism $\Phi:\mathcal{E}_1/\mathcal{O}_K\rightarrow\mathcal{E}_2/\mathcal{O}_K$ (resp. $\Psi:\mathcal{E}_2/\mathcal{O}_K\rightarrow\mathcal{E}_1/\mathcal{O}_K$) over $\mathcal{O}_K$. Let $\widetilde{E_i}/k$ be the special fiber of $\mathcal{E}_i/\mathcal{O}_K$ for $i=1,2$. Let $\widetilde{\phi}: \widetilde{E_1}/k\rightarrow \widetilde{E_2}/k$ (resp. $\widetilde{\psi}: \widetilde{E_2}/k\rightarrow \widetilde{E_1}/k$) be the induced morphism of $\Phi$ (resp. $\Psi$) on the special fibers. Then $\widetilde{\psi}\circ \widetilde{\phi}$ is the multiplication by $d$ map on $\widetilde{E_1}/k$. In particular, $\widetilde{\phi}$ is not the constant map. Let $[p_i]$ be the multiplication by $p$ map on $\widetilde{E_i}/k$ for $i=1,2$. Consider the following commutative diagram.
	
	$$
	\begin{CD}
		\widetilde{E_1}	@>\widetilde{\phi}>> \widetilde{E_2}\\
		@V[p_1]VV			@VV[p_2]V\\
		\widetilde{E_1}	@>\widetilde{\phi}>> 	 \widetilde{E_2}\\
	\end{CD}
	$$
	
	Let $\operatorname{deg}_i$ denotes the inseparable degree of a morphism. Then $\operatorname{deg}_i(\widetilde{\phi}\circ[p_1])=\operatorname{deg}_i([p_2]\circ\widetilde{\phi})$. So $\operatorname{deg}_i([p_1])=\operatorname{deg}_i([p_2])$. Meanwhile, an elliptic curve over a field of characteristic $p$ is supersingular if and only if $\operatorname{deg}_i([p])=p^2$ (see \cite{SilvermanArith} V.3 Theorem 3.1). Hence the conclusion follows. 
\end{proof}

\begin{corollary}\label{cor.sameIn}
	Let $\mathcal{O}_K$ be as in \ref{em.intro}. Let $E_1/K$ and $E_2/K$ be isogenous elliptic curves with additive reduction. Assume that $E_1/K$ and $E_2/K$ have good supersingular reduction after a quadratic extension $L/K$. Assume that $v(j(E_1))> 4s_{L/K}-4$ and $v(j(E_2))> 4s_{L/K}-4$. Then $E_1/K$ and $E_2/K$ have the same reduction type. 
	
	The latter condition is automatically satisfied when $v(2)=1$. So in this case $E_1/K$ and $E_2/K$ have the same reduction type and the reduction type is either $\mathrm{II}$ or $\mathrm{II}^{\ast}$.

\end{corollary}

\begin{proof}
	It is proved in \cite{Grothendieck}, IX, Cor. 2.2.7, that the extension $L/K$ is invariant under isogeny. Having supersingular reduction is also invariant under isogeny by Lemma \ref{le.supersing}. Therefore, if $E_{1,L}/L$ has good supersingular reduction where $L/K$ is a quadratic extension, then $E_{2,L}/L$ has good supersingular reduction. By Theorem \ref{thm.supred} (b), $E_1/K$ and $E_2/K$ have the same reduction type, since the reduction type only depends on $s_{L/K}$ under the assumption.
	
	If $K$ has mixed characteristic $2$ and $v(2)=1$, then the condition $v(j(E_1))> 4s_{L/K}-4$ and $v(j(E_2))> 4s_{L/K}-4$ holds automatically by Corollary \ref{cor.v2_sm} (a) and Theorem \ref{thm.supred}. So $E_1/K$ and $E_2/K$ have the same reduction type. The reduction type is either $\mathrm{II}$ or $\mathrm{II}^{\ast}$ by Corollary \ref{cor.v2_sm} (a).
\end{proof}

\begin{remark}\label{rk.counteregisog}

	Let $\mathcal{O}_K$ be as in \ref{em.intro}. The following example shows that the reduction type of  isogenous elliptic curves over $K$ can be different if $v(2)\ge 2$.

	Let $E_1'/\mathbb{Q}_2^{unr}$ be the elliptic curve defined by $y^2=x^3+1$. Let $K/\mathbb{Q}_2^{unr}$ be the extension given by $x^3-2$ so that $v_K(2)=3$. Then the base change $E_{1,K}'/K$ has good reduction. Let $L/K$ be a quadratic extension such that $s_{L/K}=2v_K(2)$ (see Lemma \ref{le.slkbound}). Assume that $L=K(\sqrt{D})$ for some $D\in K$. Let $E_1/K$ be the quadratic twist of $E_{1,K}'/K$ by the extension $L/K$. Then $E_1/K$ can be given by $y^2=x^3+D^3$. The group $E_1(K)$ has an order $2$ subgroup $T$ generated by the point $(-D,0)$.

	Applying Vélu's formulas \cite{Velu} to $E_1/K$ with respect to $T$, we get an elliptic curve $E_2/K$ defined by $y^2=x^3-15D^2x+22D^3$ and a $2$-isogeny $\phi$ from $E_1/K$ to $E_2/K$ such that $\operatorname{Ker}(\phi)(K)=T$. The $j$-invariants are $j(E_1)=0$ and $j(E_2)=2^4 3^3 5^3$. We have $v_K(j(E_1))> 4s_{L/K}-4$ and hence $E_1/K$ has reduction type $\mathrm{II},\mathrm{I}^{\ast}_0$ or $\mathrm{II}^{\ast}$ by Theorem \ref{thm.supred} (b). However, $v_K(j(E_2))=4v_K(2)\le 8v_K(2)-4=4s_{L/K}-4$ and hence $E_2/K$ has reduction type $\mathrm{I}^{\ast}_{4v_K(2)}$ by Theorem \ref{thm.supred} (a). So the reduction types of $E_1/K$ and $E_2/K$ are different in this case.

\end{remark}

We recall the following result when $K$ is a finite extension of $\mathbb{Q}_p$.

\begin{theorem}[T. Dokchitser and V. Dokchitser \cite{Dokchitser-Dokchitser} Theorem 5.4 (1)]\label{thm.Dok}		Let $K$ be a finite extension of $\mathbb{Q}_p$ and let $\phi: E_1\rightarrow E_2$ be an isogeny of elliptic curves over $K$ of prime degree $\ell\ne p$. Then the Kodaira types of $E_1/K$ and $E_2/K$ are the same.
\end{theorem}

We can use this Theorem in the case $p=2$ to obtain:

\begin{corollary}\label{cor.samejinv}
	Let $K$ be a finite extension of $\mathbb{Q}_2$ and let $\phi: E_1\rightarrow E_2$ be an isogeny over $K$ of odd degree. Assume that $E_1/K$ and $E_2/K$ have additive reduction and good supersingular reduction after a quadratic extension $L/K$. If $v(j(E_1))\le 4s_{L/K}-4$, then $v(j(E_1))=v(j(E_2))$.
\end{corollary}

\begin{proof}
	Since $v(j(E_1))\le 4s_{L/K}-4$, Theorem \ref{thm.supred} (a) shows that $E_1/K$ has reduction type $\mathrm{I}^{\ast}_{-v(j(E_1))+4s_{L/K}}$. So $E_2/K$ has reduction type $\mathrm{I}^{\ast}_{-v(j(E_1))+4s_{L/K}}$ by Theorem \ref{thm.Dok}. Theorem \ref{thm.supred} (a) implies that $v(j(E_2))\le 4s_{L/K}-4$ and $E_2/K$ has reduction type $\mathrm{I}^{\ast}_{-v(j(E_2))+4s_{L/K}}$. Therefore $v(j(E_1))=v(j(E_2))$.
\end{proof}

To determine the reduction types of isogenous elliptic curves over any discrete valuation field $K$ of residue characteristic $2$ when $v(2)\ge 2$, we propose the following conjecture and Proposition \ref{prop.sametype}. Recall that $s_{L/K}\le 2v(2)$ (see \ref{le.slkbound}). Conjecture \ref{conj.eqjinv} generalizes Corollary \ref{cor.samejinv}.

\begin{conjecture}\label{conj.eqjinv}
	Let $\mathcal{O}_K$ be as in \ref{em.intro}. Let $E_1/K$ and $E_2/K$ be elliptic curves and assume that there exists an isogeny of odd degree $d$ from $E_1/K$ to $E_2/K$. Assume that $0< v(j(E_1))< 15v(2)$. Then $ v(j(E_1))=v(j(E_2))$. 
\end{conjecture}

\begin{proposition}[\cite{Wang_mod}, Proposition $\refverification$]\label{prop.vj}
	Conjecture \ref{conj.eqjinv} is true for all odd $d\le\bd$.
\end{proposition}

Proposition \ref{prop.vj} is proved through a close examination of the coefficients of the classical modular polynomials. Our next proposition suggests that a version of Theorem \ref{thm.Dok} might hold in the equicharacteristic case.

\begin{proposition}\label{prop.sametype}
	Let $\mathcal{O}_K$ be as in \ref{em.intro}. Let $E_1/K$ and $E_2/K$ be elliptic curves and assume that there exists an isogeny of odd degree $d$ from $E_1/K$ to $E_2/K$. Assume that $E_1/K$ and $E_2/K$ have additive reduction and good supersingular reduction after a quadratic extension.
	
	$\quad$Assume that Conjecture \ref{conj.eqjinv} is true. Then the reduction types of $E_1/K$ and $E_2/K$ are the same.
\end{proposition}

\begin{proof}
	Recall that the quadratic extension $L/K$ to semistability is the same for $E_1/K$ and $E_2/K$. Let $j_m:=j(E_m)$ for $m=1,2$.
	
	If $v(j_m)\ge 4s_{L/K}-4$ for $m=1$ and $2$, then $E_1/K$ and $E_2/K$ have the same reduction type by Corollary \ref{cor.sameIn}.
	
	If $v(j_m)< 4s_{L/K}-4$ for $m=1$ or $2$, then $v(j_m)\le 12v(2)$ by Proposition \ref{prop.minext-twist} (b). Because we assumed that Conjecture \ref{conj.eqjinv} is true, it follows that $v(j_1)=v(j_2)$. So the reduction types of $E_1/K$ and $E_2/K$ are the same by Theorem \ref{thm.supred} (a) since the reduction type only depends on $s_{L/K}$ and the valuation of the $j$-invariant. 
\end{proof}

\begin{remark}
	The following example shows that the upper bound in Conjecture \ref{conj.eqjinv} is sharp in the degree $3$ case. Let $E_1/\mathbb{Q}$ and $E_2/\mathbb{Q}$ be elliptic curves with LMFDB labels 27.a1	and 27.a3, respectively. There is a $3$-isogeny from $E_1/\mathbb{Q}$ to $E_2/\mathbb{Q}$. Meanwhile, $j(E_1)=-2^{15}\cdot3\cdot5^3$ and $j(E_2)=0$.
\end{remark}

We now consider the reduction types of elliptic curves related by a $2$-isogeny.

\begin{proposition}\label{prop.2isog}
	Let $\mathcal{O}_K$ be as in \ref{em.intro} with mixed characteristic $2$. Let $E_1/K$ and $E_2/K$ be elliptic curves and assume that there exists a $2$-isogeny from $E_1/K$ to $E_2/K$. Let $j_1:=j(E_1)$ and $j_2:=j(E_2)$. Then the following is true.
	\begin{enumerate}[label=(\roman*)]
		\item If $v(j_1)=v(j_2)>0$, then $v(j_1)=v(j_2)=6v(2)$.
		\item If $0<v(j_1)<v(j_2)$, then $v(j_1)<6v(2)$. Moreover:
		
		\begin{enumerate}[label=(\alph*)]
			\item If $v(j_1)<4v(2)$, then $v(j_2)=2v(j_1)$.
			\item If $4v(2)<v(j_1)<6v(2)$, then $v(j_2)=12v(2)-v(j_1)$.
			\item If $v(j_1)=4v(2)$, then $v(j_2)= 8v(2)+3r$ for some $r\in\mathbb{Z}_{\ge 0}$. 
		\end{enumerate}
		\item Assume that $K=\mathbb{Q}_2^{\operatorname{unr}}$, $j_1,j_2\in \mathbb{Q}_2\subset K$ and $v(j_1)=4v(2)<v(j_2)$. Then $v(j_2)= 8v(2)+3r$ for some $r\in\mathbb{Z}_{\ge 1}$.
	\end{enumerate}
	
\end{proposition}

\begin{proof}
	Recall that 
	\begin{align*}
		\Phi_2(X,Y)=X^3 -X^2Y^2 + 1488X^2Y -162000X^2 + 1488XY^2+ 40773375XY \\
		\pushright{ + 8748000000X+ Y^3 -162000Y^2 + 8748000000Y -157464000000000}.
	\end{align*}
	
	See \cite{Fricke} (8) and (9) on page 372. Let $C$ be the affine curve defined by $\Phi_2(X,Y)=0$. Let $\overline{\Phi}_2(X,Y,Z)$ be the homogenization of $\Phi_2(X,Y)$ and let $\overline{C}$ be the projective curve defined by $\overline{\Phi}_2(X,Y,Z)=0$. Under the natural identification, we have that $\overline{C}(\overline{K})\backslash\{(1:0:0),(0:1:0)\}=C(\overline{K})$.

	The curve $\overline{C}$ has the following parametrization (\cite{Maier}, Table 7):
	
	\begin{align*}
		f:\mathbb{P}^1(\overline{K})&\longtwoheadrightarrow \overline{C}(\overline{K})\\
		(t_1:t_2)&\longmapsto ((t_1+16t_2)^3t_1:(t_1+256t_2)^3t_2:t_1^2t_2^2).
	\end{align*}
	Note that $f((1:0))=(1:0:0)$ and $f((0:1))=(0:1:0)$. Under the natural identification, $f$ restricts to the following surjective map:
	\begin{align*}
		f_0:\overline{K}\setminus\{0\}&\longtwoheadrightarrow C(\overline{K})\\
		t&\longmapsto ((t+16)^3/t,(t+256)^3/t^2).
	\end{align*}

	Note that if $f_0(t)=(X,Y)$, then $f_0(2^{12}/t)=(Y,X)$. Let us now prove the statements (i)-(iii) in the proposition. Without loss of generality, assume that $0<v(j_1)\le v(j_2)$. Since $\Phi_2(j_1,j_2)=0$, we have that $(j_1,j_2)=((t+16)^3/t,(t+256)^3/t^2)$ or $(j_1,j_2)=((t+256)^3/t^2,(t+16)^3/t)$ for some $t\in\overline{K}\backslash\{0\}$ satisfying $v(t)\le 6v(2)$. Let $X:=(t+16)^3/t$, and $Y:=(t+256)^3/t^2$. By abuse of notation, we also use $v$ to denote the unique extension of the valuation on $K$ to any finite extension $F$ of $K$ so that $v(\pi_F)=\frac{1}{[F:K]}v(\pi_K)$. The following are the possible cases.

	\begin{enumerate}[wide, labelwidth=!, labelindent=16pt, label=(\arabic*)]
		\item 	If $v(t)<4v(2)$, then $v(X)=2v(t)$ and $v(Y)=v(t)$. Because $v(j_1), v(j_2)> 0$, we know that $v(t)>0$. Since $v(j_1)\le v(j_2)$, we have that $(j_1,j_2)=(Y,X)$ and $v(j_2)=2v(j_1)$.
		\item 	
		Assume that $4v(2)<v(t)\le 6v(2)$. Then $v(X)=12v(2)-v(t)$ and $v(Y)=v(t)$. If $4v(2)<v(t)<6v(2)$, then $(j_1,j_2)=(X,Y)$ since $v(j_1)\le v(j_2)$. It follows that $v(j_2)=12v(2)-v(j_1)$. If $v(t)=6v(2)$, then $(j_1,j_2)=(X,Y)$ or $(j_1,j_2)=(Y,X)$. In either case, we have that $v(j_1)=v(j_2)=6v(2)$.
		
		\item 
		If $v(t)=4v(2)$, then $v(Y)=v(t)=4v(2)$, and $v(X)=3v(t+16)-v(t)=3v(2^{-4}(t+16))+8v(2)$. Thus $(j_1,j_2)=(Y,X)$. Therefore $v(j_1)=4v(2)$ and $v(j_2)=8v(2)+3r$ for some $r\in\mathbb{Z}_{\ge 0}$. 
	\end{enumerate}
	
	If $v(j_1)=v(j_2)>0$, then we are in Case (2) and $v(j_1)=v(j_2)=6v(2)$. This proves (i). If $0<v(j_1)<v(j_2)$, In each of the cases (1)-(3), $v(j_1)<6v(2)$. If $v(j_1)<4v(2)$, then we are in Case (1) and $v(j_2)=2v(j_1)$. If $4v(2)<v(j_1)<6v(2)$, then we are in Case (2) and $v(j_2)=12v(2)-v(j_1)$.
	If $v(j_1)=4v(2)$, then we are in Case (3) and $v(j_2)= 8v(2)+3r$ for some $r\in\mathbb{Z}_{\ge 0}$. This finishes the proof of (ii).
	
	We now prove (iii). Assume that $K=\mathbb{Q}_2^{\operatorname{unr}}$ and $j_1,j_2\in \mathbb{Q}_2\subset K$ and $v(j_1)=4v(2)<v(j_2)$. Then we are in Case (3) in the proof of (i) and (ii), i.e., there exists a $t\in\overline{K}\backslash\{0\}$ satisfying $v(t)= 4v(2)$ and $(j_1,j_2)=((t+256)^3/t^2,(t+16)^3/t)$. We claim that $t\in\mathbb{Q}_2\subset K$. We prove it by contradiction. Assume that $t\notin\mathbb{Q}_2$.
	Let $\alpha:=(t+256)^3-j_1 t^2$ and $\beta:=(t+16)^3-j_2 t$. Then $\alpha=\beta=0$. It follows that
	\begin{align*}
		4095(\alpha-\beta)&+(720-j_1)(\alpha-16^3\beta)/t
		=(j_1^2 + 195120j_1 + 4095j_2 + 660960000)t\\
		&- 4096j_1j_2 + 2949120j_1 + 
		2949120j_2 + 66562560000
		=0
	\end{align*}
	Because $t\notin \mathbb{Q}_2$, we derive that 
	\begin{equation}\label{eq.coef_t}
		j_1^2 + 195120j_1 + 4095j_2 + 660960000=0
	\end{equation}
	and 
	\begin{equation}\label{eq.const}
		- 4096j_1j_2 + 2949120j_1 + 
		2949120j_2 + 66562560000=0.
	\end{equation}
	Solve for $j_2$ from \eqref{eq.coef_t} and plug it into \eqref{eq.const}, we get that
	$$	(j_1 + 3375)(j_1^2 + 191025j_1 - 121287375)=0.$$
	Therefore $j_1=-3375$ or $j_1$ is a solution of $j_1^2 + 191025j_1 - 121287375=0$. In either case, $v(j_1)=0$. This is a contradiction. Hence $t\in\mathbb{Q}_2$. It follows that $2^{-4}t$ reduces to $1$ in the residue field $k$. So $v(2^{-4}(t+16))=v(2^{-4}t+1)>0$. From Case (3), we conclude that $v(j_2)=3v(2^{-4}(t+16))+8v(2)$ with $v(2^{-4}(t+16))>0$. Hence part (iii) is true.
\end{proof}

\begin{proposition}\label{prop.sametype2}
	Let $\mathcal{O}_K$ be as in \ref{em.intro} with mixed characteristic $2$. Let $E_1/K$ and $E_2/K$ be elliptic curves and assume that there exists a $2$-isogeny from $E_1/K$ to $E_2/K$. Assume that $E_1/K$ and $E_2/K$ have additive reduction and good supersingular reduction after a quadratic extension $L/K$. Then $E_1/K$ and $E_2/K$ have the same Kodaira type if and only if either $\operatorname{min}(v(j(E_1)),v(j(E_2)))>4s_{L/K}-4$ or $v(j(E_1))=v(j(E_2))=6v(2)$.
\end{proposition}

\begin{proof}
	Let $j_1:=j(E_1)$ and $j_2:=j(E_2)$. Without loss of generality, assume that $v(j_1)\le v(j_2)$. If $v(j_1)>4s_{L/K}-4$ or $v(j_1)=v(j_2)=6v(2)$, then $E_1/K$ and $E_2/K$ have the same Kodaira type by Theorem \ref{thm.supred}.

	Conversely, assume that $E_1/K$ and $E_2/K$ have the same Kodaira type. Assume that $v(j_1)\le 4s_{L/K}-4$. Then $E_1/K$ has Kodaira type $\mathrm{I}^{\ast}_{-v(j_1)+4s_{L/K}}$ by Theorem \ref{thm.supred} (a). If $v(j_2)>4s_{L/K}-4$, then $E_2/K$ has Kodaira type $\mathrm{II}, \mathrm{I}^{\ast}_0$ or $\mathrm{II}^{\ast}$ by Theorem \ref{thm.supred} (b). This is a contradiction. If $v(j_2)\le 4s_{L/K}-4$, then $E_2/K$ has Kodaira type $\mathrm{I}^{\ast}_{-v(j_2)+4s_{L/K}}$ by Theorem \ref{thm.supred} (a). Hence $v(j_1)=v(j_2)$. By Proposition \ref{prop.2isog} (i), $v(j_1)=v(j_2)=6v(2)$.
\end{proof}

\begin{remark}
	Proposition \ref{prop.sametype2} suggests that when the wild ramification is large, the reduction types are more likely to be different.
\end{remark}

\bibliographystyle{plain}
\nocite{*}
\bibliography{ref_red.bib}

\end{document}